\documentclass[reqno]{amsart}
\usepackage{amssymb, amscd, amsfonts}
\usepackage[mathscr]{eucal}
\usepackage{graphicx}
\usepackage{amscd}

\newtheorem{theorem}{Theorem}[section]
\newtheorem{lemma}[theorem]{Lemma}
\newtheorem{corollary}[theorem]{Corollary}
\newtheorem{proposition}[theorem]{Proposition}

\theoremstyle{definition}
\newtheorem{definition}[theorem]{Definition}

\newtheorem{remark}[theorem]{Remark}
\newtheorem{problem}[theorem]{Problem}

\newcommand{\remove}[1]{}
\newcommand{\luk}{\L u\-k\-a\-s\-ie\-wicz}

\DeclareMathOperator{\ide}{{\rm id}_{\interval}}

\DeclareMathOperator{\McNone}{{\mathcal M}([0,1])}
\DeclareMathOperator{\interval}{[0,1]}
\DeclareMathOperator{\prim}{\rm Prim}

\DeclareMathOperator{\lex}{+_{\rm lex}}
\DeclareMathOperator{\gen}{\rm gen}

\title[Word problems in Elliott involutive monoids]
{Word problems in  Elliott  monoids}

\author[D. Mundici]{Daniele Mundici$^\dag$}
\address[D. Mundici]{Department of Mathematics
and Computer Science ``Ulisse Dini'',
University of Florence, 
viale Morgagni 67/A,
50134 Florence,  Italy}
\email{mundici@math.unifi.it}



 \date{\today}
 
\begin{document}

  \maketitle
  
   \hfill{\it to Saunders Mac Lane, in memoriam}

\begin{abstract} 
Algorithmic issues  concerning Elliott local 
semigroups  are 
seldom considered in the literature,
although  these combinatorial 
structures completely classify 
 AF algebras. In general, the addition  operation of 
an Elliott local
semigroup is {\it partial}, but  for every AF algebra 
$\mathfrak B$  whose  Murray-von Neumann order
of projections is a lattice,   
this operation is  uniquely extendible
to the addition of an involutive  monoid $E(\mathfrak B)$.
Let $\mathfrak M_1$ be the
Farey  AF algebra
 introduced by the present author
  in 1988 and rediscovered by F. Boca in 2008.
The freeness properties of the involutive monoid
$E(\mathfrak M_1)$ yield
a natural  word problem for every AF algebra  $\mathfrak B$ 
with singly generated $E(\mathfrak B)$, because 
$\mathfrak B$ is automatically a quotient of $\mathfrak M_1$.
 Given two formulas  $\phi$ and $\psi$ 
 in the language of involutive
 monoids, the problem asks to decide whether
 $\phi$ and $\psi$   code the same 
  equivalence of projections of $\mathfrak B$.
This mimics the classical definition 
of the word problem of a group presented  by 
 generators and relations.   
We  show   that the
word problem  of  
$\mathfrak M_1$  is solvable in polynomial time,
and so is the word problem of the
Behnke-Leptin algebras $\mathcal A_{n,k}$, 
and of the Effros-Shen algebras  
$\mathfrak F_{\theta}$, for
 $\theta\in \interval\setminus \mathbb Q$ 
a real algebraic
number, or $\theta = 1/e$. 
 We construct a  quotient of $\mathfrak M_1$
 having a  G\"odel incomplete
  word problem, and show that  no
  primitive
  quotient of $\mathfrak M_1$ is
  G\"odel incomplete.
\end{abstract}

\section {Foreword}
%
For quantum  physical systems  $\mathcal S$ 
with   infinitely many degrees of freedom, 
von Neumann's uniqueness theorem does not hold
in general, and  {\it the}
Hilbert space of  $\mathcal S$ does not exist.  
%
Infinite  systems
typically occur in quantum statistical mechanics,
\cite{sew}. 
Glimm's  UHF  algebras
 yield the standard tool for the algebraization of
  quantum spin systems  \cite{bra, brarob, emc, sew}.
For instance, the Canonical Anticommutation
Relation (CAR) algebra provides a
mathematical description of the ideal Fermi gas.
Eventually,  the following general definition
emerged:
An   {\it AF  algebra}  $\mathfrak B$ is the
norm closure of the union of an ascending sequence of finite
dimensional  C*-algebras,   all with the same unit,
 \cite{bra, eff, goo-shiva}.

When the Murray-von Neumann order of projections in
$\mathfrak B$  is a lattice we say that $\mathfrak B$ 
is an  {\it  $AF\ell$ algebra. }

Quite a few  classes of AF algebras existing in the literature
turn out to be AF$\ell$ algebras, well beyond the most 
elementary   examples
 given by commutative AF algebras and finite
dimensional C$^*$-algebras.
An  interesting class
of AF$\ell$ algebras is given by 
AF algebras with
comparability of projections (in the sense of
Murray-von Neumann). This includes 
the CAR algebra and, more generally,
 all of Glimm's UHF algebras, \cite{eff, brarob}.
AF algebras with
comparability of projections also include the
Effros-Shen C*-algebras   
$\mathfrak F_{\theta}$ for irrational $\theta\in\interval$, 
\cite[p.65]{eff}, which play  an
interesting role in topological dynamics,
\cite{BMT87, bratteli1, bratteli2, pimvoi}.
Also the  Behncke-Leptin
C*-algebras $\mathcal A_{m,n}$ with a two-point dual
\cite{behlep} have comparability of projections.  
Other examples of AF$\ell$ algebras are given by
continuous trace   AF algebras,  
 liminary C*-algebras with boolean 
primitive spectrum, \cite{cigellmun},  (see 
\cite{mun-padova}  for an interesting particular case),
and  AF algebras with a directed set of finite
dimensional *-subalgebras, \cite{laz1,laz2}.

The  ``universal''
  AF algebra $\mathfrak M$
of  \cite[\S 8]{mun-jfa} is an AF$\ell$ algebra:
every  AF algebra  with comparability
of projections is a quotient of $\mathfrak M$
by a   primitive essential ideal,
 \cite[Corollary 8.7]{mun-jfa}. Also,  every 
 (possibly non-unital) AF algebra may be embedded into
a quotient of $\mathfrak M$, \cite[Remark 8.9]{mun-jfa}.
In this paper we will consider the 
Farey  AF$\ell$  algebra $\mathfrak M_{1}$
introduced in 
 \cite{mun-adv}. By \cite{pimvoi}, 
 every irrational rotation C$^*$-algebra 
 is embeddable into some (Effros-Shen) simple quotient of  
 $\mathfrak M_1, $ \cite[Theorem 3.1 (ii)]{mun-adv}.
$\mathfrak M_1$, in turn,
  is embeddable into
  Glimm's universal UHF algebra,
   \cite[Theorem 1.5]{mun-milan}.
 As shown in \cite{mun-lincei, mun-milan}, 
 the AF algebra $\mathfrak A$ more recently
 considered by   Boca \cite{boc}  coincides with 
$\mathfrak M_1$. For the many interesting properties and  
 applications of $\mathfrak M_1$  see 
 \cite{agu, boc, eck,
 mun-adv,  mun-lincei, 
	mun-milan}.

As we will see in Theorem  \ref{theorem:multi},
a main property of every AF$\ell$ algebra  $\mathfrak B$ is that
the partial addition in its Elliott local semigroup 
is canonically extendible to the addition of
 an involutive  monoid $E(\mathfrak B)$.
We are interested in the computational complexity
of the word   problem of  $\mathfrak B$,
asking whether two
formulas in the  language of involutive monoids
denote the same Murray-von Neumann equivalence class of
projections of  $\mathfrak B$.
By Elliott's classification,  this
 is the equivalent counterpart for $\mathfrak B$
of the classical word problem of a
group presented by generators and relations, \cite{man}.

In Theorem \ref{theorem:polytime-emmone} we establish the
polynomial time complexity of the
word problem 
of  ${\mathfrak{M}_1}$. 
Since
$E({\mathfrak{M}_1})$ is free, the
word problem can be naturally 
posed for  all AF$\ell$ algebras  $\mathfrak Q$
with singly generated
 $E(\mathfrak Q)$,  because the preservation properties
 of $K_0$ ensure that $\mathfrak Q$ is a quotient
 of $\mathfrak M_1$, just as 
$E(\mathfrak Q)$ is a  quotient
of  $E({\mathfrak{M}_1})$.
In Theorem
\ref{theorem:polytime-bl} and Corollary
\ref{corollary:polytime-other-bl}  
we prove  the polynomial time complexity of the word problem
of  Behnke-Leptin C$^*$ algebras
$\mathcal A_{m,n}$.  
In Theorem \ref{theorem:polytime-ef-cf}
and Corollaries 
   \ref{corollary:quadratic}
   and  \ref{corollary:polytime-ef-cut},   a large set 
 of  Effros-Shen algebras  
$\mathfrak F_{\theta}$ is shown to have a polynomial
time word problem. This includes the case  when
$\theta = 1/e$, or 
 $\theta\in \interval\setminus \mathbb Q$ is  
a real algebraic
number. 
  In a final section a quotient of
$\mathfrak M_1$ is constructed having a  
G\"odel incomplete word problem. It is also shown
that  G\"odel incompleteness does not affect primitive quotients.

Familiarity is required with Elliott's classification
of AF algebras and  
 $K_0,$  \cite{eff,  ell, ell-lnm, goo-shiva},  
lattice-ordered abelian groups
\cite{bigkeiwol}, and
the basic properties of
the equational class of
 Elliott involutive monoids arising from  AF$\ell$ algebras---better
known as MV-algebras,
\cite{cha58, cha59,  cigdotmun,  mun11}. 
For the rudiments
of diophantine approximation applied
to  Effros-Shen algebras  in this paper, 
 we refer to \cite{lan}.
For background on Turing machines and
algorithmic  complexity see  
\cite{man, mon, munsie}. 
For complexity theoretic issues
concerning  real numbers we refer
to \cite{ker} and \cite{mis}.

\section{AF algebras with a Murray-von Neumann lattice of
projections}

 Elliott  local  semigroup 
 \cite{ell}, is  a ring invariant consisting of equivalence classes of idempotents,   together  with the partially defined addition of equivalence classes of orthogonal idempotents,  two idempotents    $i$    and    $j$  being  {\it equivalent}   iff   
 $xy = i$    and   $ yx  =  j$    
 for some   $x$   and  $ y$  in the ring.
	 
As explained in \cite[Appendix]{ell} and in \cite[p. 32]{ell-lnm},   for every AF algebra $\mathfrak B$  the invariant is complete.  
Elements of Elliott local semigroup  $L(\mathfrak B)$  boil down to 
Murray von Neumann equivalence classes of projections,  with the partially defined operation   
$[p] + [q] =   [p + q] $,   whenever     
$p$     and    $ q$     are orthogonal.

For any two equivalence classes    $[p]$   and 
$ [q]$   in   $L(\mathfrak B)$,   we write   $[p]  \leq  [q]$  iff   $p$   is equivalent to a subprojection of  $q$.   
This order is known as the Murray von Neumann order   
of (projections in)  $\mathfrak B$.
Elliott's partially defined addition  +  is associative, commutative, and monotone:  if   $[p]  \leq  [q]$   and  $[r]  \leq  [s]$   then    
$ [p] + [r]  \leq  [q]  +  [s] $,   whenever $ [q]  +  [s]$    is defined.

Since every AF algebra $\mathfrak B$ in this paper
has a unit element, 
 the underlying order of the $K_0$-group of $\mathfrak B$,
 as well as the unit element  $[1_\mathfrak B],$
 will be tacitly understood, whenever no confusion
 may arise.  
So we will usually  write 
$$
\mbox{
$K_0(\mathfrak B)$ instead of
   $(K_0(\mathfrak B),
[1_\mathfrak B])$
and  $(K_0(\mathfrak B),
K_0^+(\mathfrak B),
[1_\mathfrak B])$.}
$$ 
By a {\it unital $\ell$-group} we mean an abelian
lattice ordered group with a distinguished 
(strong) unit $u>0$, \cite{bigkeiwol, goo-ams}.

 \smallskip
The theory of  AF$\ell$ algebras 
is grounded in  the following result:

\begin{theorem}
\label{theorem:multi}  

Let $\mathfrak B$ be an AF algebra and $L(\mathfrak B)$  its
 (partially ordered)  local  semigroup.
 
\medskip
(i)  \cite{munpan}
Elliott's partially defined addition $+$ in 
$L(\mathfrak B)$   has 
{\em at most one}
 extension to an associative, commutative, 
monotone operation  $\oplus \colon
L(\mathfrak B)^2\to L(\mathfrak B)$ 
satisfying the following condition:
For each projection  $p \in  \mathfrak B$, 
$ [1_\mathfrak B - p ]$   
 is the smallest element $[q]\in L( \mathfrak B)$ with    $[p]+$
 $[q]=[1_\mathfrak B].$ 
The
 unique semigroup  $(S(\mathfrak B),\oplus)$
expanding  the Elliott  local  semigroup 
$L(\mathfrak B)$  exists iff    $\mathfrak B$   
is an AF$\ell$ algebra.

\medskip
(ii)  \cite{ell, ell-lnm}    Let    $\mathfrak B_1$   and    
$\mathfrak B_2$    be  AF$\ell$ algebras.    
For  each     $i = 1, 2$    let      $\oplus_i$    
be the  extension  of Elliott's addition  given by (i).    
Then the semigroups  $(S(\mathfrak B_1), \oplus_1)$    and   
 $(S(\mathfrak B_2), \oplus_2)$   are isomorphic iff so are      
$\mathfrak B_1$   and    
$\mathfrak B_2$.
 
\medskip
   (iii)   \cite{munpan} For any
   AF$\ell$ algebra 
 $\mathfrak B$   the semigroup
$(S(\mathfrak B), \oplus)$ 
has  the  structure of  a  monoid
$(E(\mathfrak B), 0, ^*,  \oplus)$ 
with an involution operation  
$[p]^* =  [1_\mathfrak B - p ]$. 
The Murray von Neumann  
lattice order  of equivalence classes of projections 
    $[p],[q]$ in    $\mathfrak B$ is definable via
    the involutive monoidal operations, by writing
    $[p] \vee [q]=([p]^*\oplus [q])^*\oplus [q]$ and
     $[p]\wedge [q]= ([p]^*\vee [q]^*)^*$   
for all $[p],[q] \in  E(\mathfrak B).$

\medskip
   (iv)   \cite{eff, ell, ell-lnm,  goo-shiva}
For any
   AF$\ell$ algebra 
 $\mathfrak B$
the dimension group $K_0(\mathfrak B)$
is a countable unital $\ell$-group.
All countable unital $\ell$-groups arise in this way.
Let    $\mathfrak B$   and    
$\mathfrak B'$    be  AF$\ell$ algebras.
Then $K_0(\mathfrak B)$ and $K_0(\mathfrak B')$
  are isomorphic as unital  
  $\ell$-groups  iff       
$\mathfrak B$   and    
$\mathfrak B'$ are isomorphic.

\medskip
(v)  \cite[Theorem 3.9]{mun-jfa}  
Up to isomorphism, the map 
 $\mathfrak B \to  (E(\mathfrak B),
 0, ^*, \oplus)$ 
   is a one-one correspondence between      AF$\ell$ algebras     and  
countable abelian monoids with a unary operation   $^*$  satisfying the equations:
$$
\mbox{$x^{**}=x,  \,\,\,\,\,\,\,   0^* \oplus x = 0^*,   \,\,\,\,\,\,
\mbox{ and }   \,\,\,\,\,\,  (x^* \oplus y)^* \oplus
 y = (y^* \oplus x)^* \oplus x.$}
 $$
Let $\Gamma$ be the categorical equivalence between
 unital  $\ell$-groups and  the class of 
 involutive monoids satisfying these equations.
Then the Elliott  involutive monoid
 $(E(\mathfrak B),
 0, ^*, \oplus)$  is isomorphic to 
 $\Gamma(K_0(\mathfrak B))$. 
 \end{theorem}

 \medskip
 
 \section{The Farey  AF algebra $\mathfrak M_1$ 
 of \cite{mun-adv} (also see \cite{boc}  and 
 \cite{mun-milan, mun-lincei})}
 \label{section:ledue}

\subsection*{First construction, \cite{mun-adv}}
For each $m=1,2,\ldots$ let $\nu(m)=1+2^{m-1}$.
The  sequence  $A_{1}, A_{2},\ldots$
of $\{0,1\}$-matrices with $\nu(m+1)$ rows
and  $\nu(m)$ columns is defined by:

{\tiny
\[A_1= \left[
\begin{matrix}
 1&0\\ 
 1&1\\ 
 0&1 
\end{matrix}
\right]  
A_2=
  \left[
\begin{matrix}
 1&0&0\\ 
 1&1&0\\ 
 0&1&0\\
 0&1&1\\
 0&0&1, 
\end{matrix}
\right]
A_3=
  \left[
\begin{matrix}
 1&0&0&0&0\\ 
 1&1&0&0&0\\ 
 0&1&0&0&0\\
 0&1&1&0&0\\
 0&0&1&0&0\\
 0&0&1&1&0\\
 0&0&0&1&0\\
 0&0&0&1&1\\
 0&0&0&0&1
\end{matrix}
\right]
A_4=
  \left[
\begin{matrix}
 1&0&0&0&0&0&0&0&0\\ 
 1&1&0&0&0&0&0&0&0\\ 
 0&1&0&0&0&0&0&0&0\\
 0&1&1&0&0&0&0&0&0\\
 0&0&1&0&0&0&0&0&0\\
 0&0&1&1&0&0&0&0&0\\
 0&0&0&1&0&0&0&0&0\\
 0&0&0&1&1&0&0&0&0\\
 0&0&0&0&1&0&0&0&0\\
 0&0&0&0&1&1&0&0&0\\
  0&0&0&0&0&1&0&0&0\\
    0&0&0&0&0&1&1&0&0\\
	0&0&0&0&0&0&1&0&0\\
	0&0&0&0&0&0&1&1&0\\
	0&0&0&0&0&0&0&1&0\\
	0&0&0&0&0&0&0&1&1\\
	0&0&0&0&0&0&0&0&1
\end{matrix}
\right]
{\ldots}
\]}

\noindent where for each  $j=1,\ldots,2^{m-1}$ the
$2j$th row of $A_{m}$ has a 1 in positions $j, j+1$
and is 0 elsewhere, the $(2j-1)$th row has a 1
in position $j$ and is 0 elsewhere,
and the last row has all 0 except a final 1. The map
$
\zeta_m \colon
x\in \mathbb Z^{\nu(m)}\mapsto A_{m}x\in  \mathbb Z^{\nu(m+1)}
$
is a one-one  homomorphism of the free abelian
group  $\mathbb Z^{\nu(m)}$ into
$\mathbb Z^{\nu(m+1)}$, which  
preserves
  the product ordering of these groups. Equipping the simplicial group
$\mathbb Z^{\nu(1)}$ with
the (strong) order unit $u_{1}=(1,1)$,  the element
$u_{t}=A_{t}A_{t-1}\cdots A_{1}u$ is an order unit
in $\mathbb Z^{\nu(t)},\,\,\,t=1,2,\ldots.$
Following \cite[pp.31,33]{mun-adv},
we denote by   $M_1$,  or $(M_1,1)$ for greater clarity,
  the unital 
  $\ell$-group of all continuous piecewise linear
  functions  $f\colon [0,1]\to \mathbb R$
  such that each linear piece of $f$ has
  integer coefficients.
 $M_1$ is
equipped with the pointwise operations $+,-,\wedge,\vee$ of
$\mathbb R$, and with the
constant function 1 as 
its distinguished order unit.
The number of linear pieces of every
element of $M_1$ is finite. Throughout, the adjective
``linear'' is to be understood in the affine sense.
In \cite[3.3]{mun-adv} it is proved that
$
\lim ((\mathbb Z^{\nu(m)},u_{m}),A_{m}) = (M_1,1), 
$
where limits are taken in the category of
partially ordered groups with order unit, \cite{goo-ams}.
In the light of Elliott's classification \cite{ell, ell-lnm} and
further $K_0$-theoretic developments, \cite{goo-shiva, eff},
(see Theorem \ref{theorem:multi}(iv))
the AF algebra  $\mathfrak  M_{1}$
was defined  in   \cite[p.33]{mun-adv}  by writing
%
$$(K_{0}(\mathfrak  M_{1}),[1_{\mathfrak  M_{1}}]) \cong 
(M_1,1).$$
{}From the unital positive homomorphisms
$A_{m}$ one may construct a Bratteli diagram 
of  $\mathfrak M_1$ following
 \cite[Chapter 2]{eff}.

\subsection*{Second construction, \cite{boc}}   
In the  Bratteli diagram of Figure \ref{figure:bratteli},
the  two top  (depth 0) vertices $v_1$ and $v_2$
have a label  1,  $\lambda(v_{01})=\lambda(v_{02})=1$.  
There are $n(d)=2^d+1$ vertices at depth $d$,  $(d=0,1,\dots).$
The label  $\lambda(v)$ of any vertex
$v$ at depth $d=1,2,\ldots$  is the sum of the labels
of the vertices at depth  $d-1$ connected to $v$ by an edge.
Identifying each vertex $v$ with the matrix algebra
$\mathbb M_{\lambda(v)}$,  for each $d=0,1,2,\ldots,$
the finite-dimensional C*-algebra  
$\mathfrak A_{d}$ is defined as the direct sum of
the   $n(d)$  matrix algebras at depth $d$.
The edges from depth $d$ to depth
$d+1$  then determine a unital *-homomorphism
 $\phi_{d}\colon \mathfrak A_{d}\to
 \mathfrak A_{d+1}$, following the standard interpretation
 of Bratteli diagrams,  \cite{bra},  \cite[Chapter 2]{eff}.
The  AF algebra
$\mathfrak A$ is defined 
as $\mathfrak A=\lim (\mathfrak A_{d},\phi_{d}).$

 \begin{figure} 
    \begin{center}                                     
    \includegraphics[height=6.9cm]{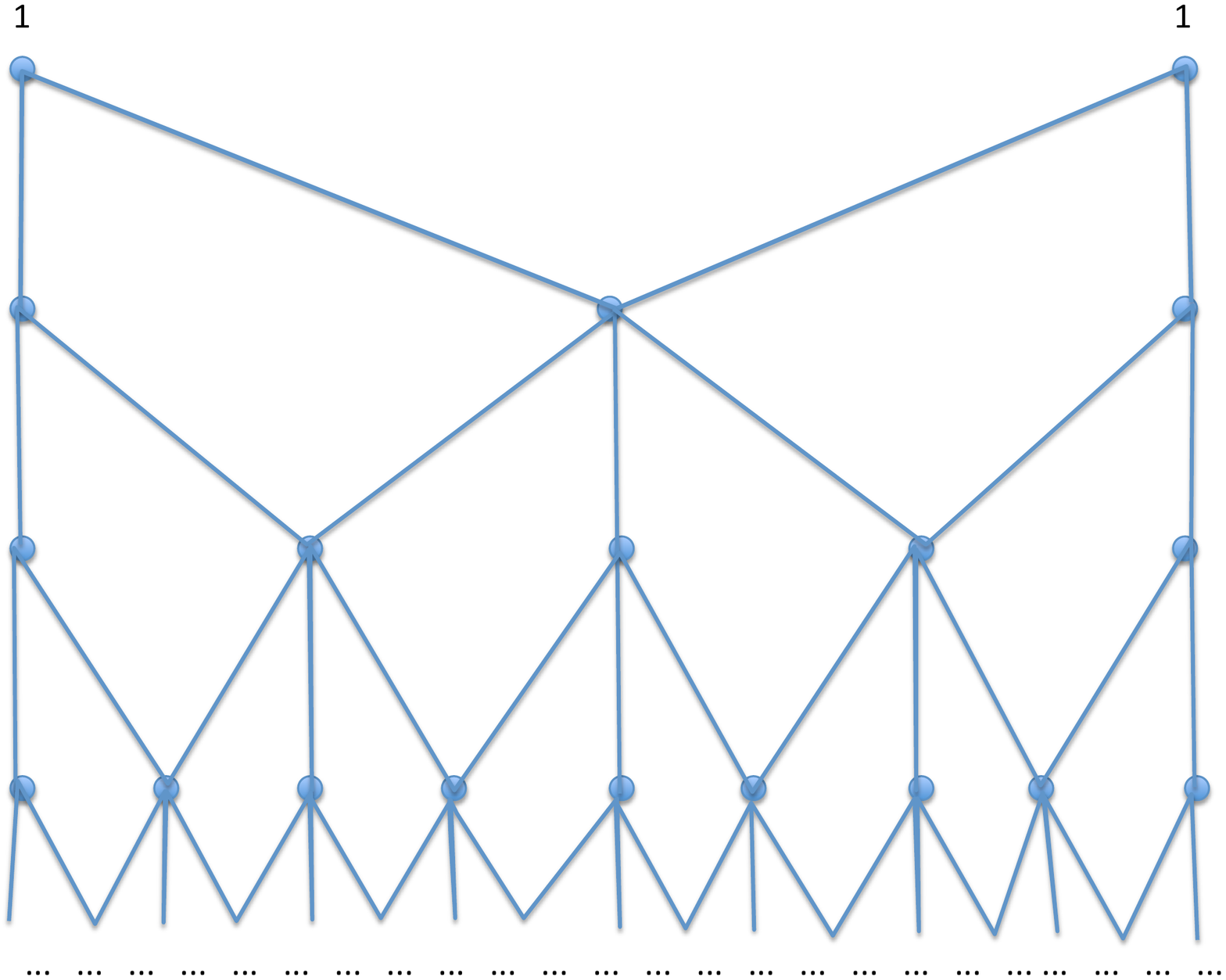}    
    \end{center}                                       
 \caption{\small The Bratteli diagram of $\mathfrak A,$
 \cite[p. 977]{boc}.}
    \label{figure:bratteli}                                                 
   \end{figure}

\begin{theorem}
{\rm \cite[p.329]{mun-lincei}, \cite[Theorem 1.1]{mun-milan}}
	\label{theorem:eguali} 
\,\,\, $\mathfrak A=\mathfrak  M_{1}.\,\,\,$
	\end{theorem}

The Farey algebra 
$\mathfrak M_1$
 has many interesting properties
 and applications,   \cite{agu,boc, eck,
 mun-adv,  mun-lincei, mun-milan}.

\begin{theorem}  
\label{theorem:generator} 
$\mathfrak M_1$ has the following properties:

\smallskip
\begin{itemize}
\item[(a)]    
The Elliott involutive monoid
$E({\mathfrak M_1})=(E({\mathfrak M_1}),0,^*,\oplus)$
is singly generated.

\smallskip
\item[(b)]  For any  two generators $[p]$ and $[q]$
of   $E({\mathfrak M_1})$, the
map  $[p]\mapsto [q]$ uniquely extends to an automorphism
of $E({\mathfrak M_1})$. 

\smallskip
\item[(c)] 
If $[p]$ is a generator of  $E({\mathfrak M_1})$,  the only other
generator is  $[1_{\mathfrak M_1}-p].$
\end{itemize}
\end{theorem}

\begin{proof}
(a)  By  Theorem \ref{theorem:multi}(v),
 the involutive monoids  
$E({\mathfrak M_1})$ and
 $\Gamma(M_1,1)$ are isomorphic.
By definition of the $\Gamma$
functor,    \cite[p.20]{mun-jfa}, the underlying set of
$\Gamma(M_1,1)$ is the unit interval of $(M_1,1)$. Therefore, 
  $\Gamma(M_1,1)$  consists of
the $\interval$-valued functions in $(M_1,1)$, 
equipped with the pointwise operations  $f^*=1-f$ and $f\oplus g=
 (f + g)\wedge 1.$
Equivalently,   $\Gamma(M_1,1)$
 is the involutive monoid  
of all piecewise linear continuous
 $\interval$-valued functions $f$ defined on $\interval$, 
 such that every linear piece of $f$ has integer coefficients.
 These are known as  one-variable {\it McNaughton functions.}
When $\Gamma(M_1,1)$ is considered independently of
$(M_1,1)$, the notation 
$\McNone$  is customary.  So
\begin{equation}
\label{equation:isomorphism}
E({\mathfrak M_1}) \cong \McNone.
\end{equation}
 As a consequence of the one-variable
 fragment of  McNaughton's theorem \cite{mac}, 
 the involutive  monoid
 $\McNone$ is   generated by the identity function
 $\ide$ on $\interval.$  Therefore,   its isomorphic
 copy  $E(\mathfrak M_1)$ is  singly generated.

\smallskip
  (b) 
    By \eqref{equation:isomorphism},
we can equivalently argue for  $\McNone$
in place of   $E({\mathfrak M_1})$.
As we have seen in (a),  $\Gamma(M_1,1)$
is the involutive monoid
$\McNone$ of all   
McNaughton functions
  $f\colon \interval\to \interval$.
Abelian  monoids with a unary operation
     $^*$  satisfying the equations 
     $x^{**}=x,  \,\,\,   0^* \oplus x = 0^*,   \,\,\,
\mbox{ and }   \,\,\,  (x^* \oplus y)^* \oplus
 y = (y^* \oplus x)^* \oplus x\,\,$ 
 of Theorem \ref{theorem:multi}(v)
  are known as {\it MV  algebras},
   \cite{cha58,  cigdotmun, mun11}. 
McNaughton's theorem \cite{mac} 
 shows that  $\McNone$  is   
the  {\it free one-generator  MV  algebra}. 
See
 \cite[Corollary 3.2.8]{cigdotmun} for a direct proof
 of the one-variable case.
%
%
%

By the {\em standard MV algebra}  
we mean the 
MV algebra obtained by  equipping
 the unit real interval  with the
involutive monoid operations  $x^*=1-x$ and
$x\oplus y=\min(1,x+y).$
Following  \cite{cigdotmun, mun11},   the standard
MV-algebra 
will be denoted $\interval$.

Chang's completeness theorem \cite{cha59}
states  that
 if an equation $\sigma$   fails in some MV  algebra
 then $\sigma$  fails  in  the standard MV algebra  $\interval$.
Since the operations
$^*$ and $\oplus$ are continuous, the proof
of Chang's theorem given in   \cite[Theorem 2.5.3]{cigdotmun}
 actually shows that  $\sigma$ also   fails in the subalgebra
  $\interval \cap \, \mathbb Q$.  
  Let $\{r_1,\dots,r_m\}$ be rationals coded by
  the terms occurring in $\sigma.$
  Let $d$ be their least common denominator.  
It follows that  $\sigma$  fails in the 
 subalgebra of $\interval\cap \, \mathbb Q$,  whose universe is 
$\{0,1/d,\dots,(d-1)/d,1\}$.

 Thus the free MV  algebra
 $\McNone$
 satisfies the hypothesis of the J\'onsson-Tarski
 theorem, \cite[Theorem 2 and remarks on pp.95 and 98]{jontar}.
 The thesis of the theorem implies  that  
 any generator $g$ of  $\McNone$ is
 a free generator.  
  By definition of a free generating set, if $g'$ is any 
  (automatically free) generator of $\McNone$, the map 
$g \mapsto g'$ uniquely extends to an automorphism
of $\McNone$. Going back to 
$E({\mathfrak M_1})\cong \McNone$ we have the desired conclusion.

\smallskip
 (c)  
 Again, we will argue for $\McNone$ in place of 
$E({\mathfrak M_1}).$
 In view of (b), it  suffices to show that
$\McNone$   has precisely
two automorphisms.  Suppose $\alpha$
is an automorphism of $\McNone$ different from the identity. 
 Let  ${\mathsf j}=\alpha(\ide).$ 
 By (b),   ${\mathsf j}\not=\ide.$   Trivially, 
 ${\mathsf j}$  is a generator of $\McNone,$ and
  is one-one, because so is $\ide$.
Since ${\mathsf j}$ has integer coefficients then
$\{{\mathsf j}(0),{\mathsf j}(1)\}=\{0,1\}$. 
From the continuity of  ${\mathsf j}$ it follows that
${\mathsf j}$  is onto $[0,1]$.  For all  
$0<z<1$, both one-sided
derivatives $\partial {\mathsf j}(z)/\partial x^{\pm}$ exist,
because   ${\mathsf j}$ is piecewise linear.
Neither derivative 
can be zero,  for otherwise ${\mathsf j}$ would not be one-one.
Since all linear pieces of ${\mathsf j}$ have integer coefficients, 
$\partial {\mathsf j}(z)/\partial x^{\pm}$ is a nonzero integer.
Therefore,  the one-one function ${\mathsf j}$ has only
one linear piece. 
Since ${\mathsf j}$ is different from $\ide,$ then for all $x\in \interval$,\,\, 
$\partial {\mathsf j}(z)/\partial x^{+} = -1$,
whence ${\mathsf j}=1-\ide.$  
Since ${\mathsf j}$ is a  generator of $\McNone$, 
  for each $g\in \McNone$, $\alpha(g)$
  is the composite function $g\circ {\mathsf j},$ whence
  $(\alpha(g))(x)=g(1-x)$ for all $x\in \interval.$ This is the only
possibility for the automorphism $\alpha$ of $\McNone$,
other than the identity automorphism.
\end{proof}

 \section{The primitive quotients of $\mathfrak M_1$}
 \label{section:primitive}

For later use, in  this
 section we  provide background material for
the study of the word problem of  
primitive quotients of $\mathfrak M_1.$

\medskip
\begin{theorem}
\label{theorem:robinson-apal}
	The complete list of prime ideals of 
the Elliott involutive monoid  $E(\mathfrak M_1)=\McNone$
  is as follows:
	
	\begin{itemize}
	\item[(i)] For each  $\xi \in [0,1],$ the maximal ideal  
 $\mathfrak j_{\xi}=\{f\in  \McNone
 \mid f(\xi)=0\}.$

 \smallskip
\item[(ii)]  $\mathfrak j_{0}^{+}=\{f\in \McNone  \mid f(0)=0\,\,\, 
\mbox{and }\,\,\,
	{\partial f}(0)/{\partial x^{+}}=0\}$.

	\smallskip
\item[(iii)]   $\mathfrak j_{1}^{-}=\{f\in \McNone \mid
f(1)=0\,\,\,
\mbox{and }\,\,\,
 {\partial f}(1)/{\partial x^{-}}=0\}$.

	\smallskip
\item[(iv)]  For every rational $0<a/b< 1$,
 $\mathfrak j_{a/b}^{\pm}=$
$\{f\in \McNone  \mid f(a/b)=0$ and 
	${\partial f}(a/b)/{\partial x^{\pm}}=0\}$.
\end{itemize}	
	\end{theorem}
	
	\begin{proof} See
	the classification of
	prime ideals in free $n$-generator 
	MV algebras,
	in  \cite[\S 2]{busmun}, for $n=1.$
	Recall that $\McNone$ is the free
	one-generator MV algebra,
	(proof of Theorem \ref{theorem:generator}(b)).
	\end{proof}
%
%
%

We will make repeated use of the following result:

 \begin{lemma}
 \label{lemma:lattice-quotient}
 Let $\mathfrak B$ be an AF algebra, and
 $(G,u)=K_0(\mathfrak B)$ its associated unital
 dimension group, \cite{eff, goo-shiva}.
 
 \medskip
 (i)  For every  ideal
 $\mathfrak I$ of $\mathfrak B$ we have the isomorphism
 $K_0(\mathfrak B/\mathfrak I)\cong 
  K_0(\mathfrak B)/K_0(\mathfrak I)$, where
  $K_0(\mathfrak I)$ is the
  image of $\mathfrak I$ in $K_0(\mathfrak B).$

 \medskip 
  (ii) The map
  $
  \mathfrak I\mapsto K_0(\mathfrak I)
  $
  is an isomorphism of the lattice of
  ideals of $\mathfrak B$ onto the lattice
  of ideals of the unital dimension group $K_0(\mathfrak B).$
  
 \medskip
  (iii) If  $\mathfrak B$ is an AF$\ell$ algebra, the map
  $
  \mathfrak I\mapsto \Gamma(K_0(\mathfrak I))
  $
   is an isomorphism of the lattice of
  ideals of $\mathfrak B$ onto the lattice
  of ideals (=kernels of homomorphisms) of  
  $E(\mathfrak B),$ where
  $ \Gamma(K_0(\mathfrak I)) $ is the intersection
  of $K_0(\mathfrak I)$ with the unit interval  $[0,u]$ 
  of the unital dimension group $(G,u).$
  
 \medskip
  (iv) 
For any    AF$\ell$ algebra $\mathfrak B$ and ideal 
$\mathfrak J$ of $\mathfrak B$, the quotient
$\mathfrak B/\mathfrak I$ is an
  AF$\ell$ algebra  and  $E(\mathfrak B/\mathfrak I)
  \cong E(\mathfrak B)/\Gamma(K_0(\mathfrak I)).$
 \end{lemma}
 
 \begin{proof} (i) 
  \cite[p.34]{ell-lnm} or 
\cite[Corollary 9.2]{eff}.
(ii)   \cite[p.196]{goo-shiva}.
(iii)  Use (ii) together with
 \cite[Theorem 7.2.2]{cigdotmun}. 
 (iv) Use  (i) together with
  \cite[Theorem 7.2.4]{cigdotmun}.   
 \end{proof}

\begin{corollary}
	\label{corollary:ell-ideals}
	The complete list of prime ideals of 
	the unital $\ell$-group
	$M_1$ is as follows:
	
	\smallskip
	
	\begin{itemize}
	\item[(i)] For each  $\xi\in [0,1],$  the maximal ideal
 $J_{\xi}=\{f\in M_1\mid f(\xi)=0\}.$

 \smallskip
\item[(ii)]  $J_{0}^{+}=\{f\in J_{0}\mid 
	{\partial f}(0)/{\partial x^{+}}=0\}$.

	\smallskip
\item[(iii)]   $J_{1}^{-}=\{f\in J_{1}\mid 
	{\partial f}(1)/{\partial x^{-}}=0\}$.

	\smallskip
\item[(iv)]  For each rational $0<a/b< 1$,
 $J_{a/b}^{\pm}=
\{f\in J_{a/b} \mid 
	{\partial f}(a/b)/{\partial x^{\pm}}=0\}$.
\end{itemize}	
	\end{corollary}
	
	\begin{proof}
Let   $(G,u)$ be a unital $\ell$-group, 
with its
associated MV-algebra  $A=\Gamma(G,u)$.
By \cite[Theorem 7.2.2]{cigdotmun},   the map
$$
\mathfrak i\in \mbox{ set of ideals of } A\mapsto
\mbox{ smallest ideal of $(G,u)$ containing }\mathfrak i
$$
  is a lattice isomorphism
of the set of ideals of $A$ onto the set of ideals of
$(G,u).$ The inverse map sends every ideal $J$ of $(G,u)$
into the intersection of $J$ with the unit interval 
$[0,u]$ of $(G,u).$
Now apply 
 Theorem \ref{theorem:robinson-apal}.
\end{proof}

\medskip
With reference to Lemma \ref{lemma:lattice-quotient}, for every
AF algebra $\mathfrak B$, let
 $\mathfrak{id}=\mathfrak{id}_\mathfrak B$ be the isomorphism
 of  the lattice of ideals of the dimension group
  $K_{0}(\mathfrak B)$
onto the lattice of ideals of $\mathfrak B$.  
Thus for
 every ideal $\mathfrak J$ of
		$\mathfrak B$, 
\begin{equation}
\label{equation:inversa}
\mathfrak{id}^{-1}(\mathfrak J) =  K_{0}(\mathfrak J)=
\mbox{ the image of $\mathfrak J$ in $K_{0}(\mathfrak B)$}.
\end{equation}
{}From  Corollary  \ref{corollary:ell-ideals}
 we immediately  obtain   the  classification of 
 the set $\prim\mathfrak M_1$ of primitive ideals of
 $\mathfrak M_1$ independently re-obtained in
   \cite[Proposition 7]{boc}: 

\begin{corollary}
	\label{corollary:primitive-ideals}
	The complete list of the set
	$\prim\mathfrak M_1$ of 
	primitive ideals of $\mathfrak M_1$
	is as follows:

	\begin{enumerate}
	\item[(i)] For each  $\xi\in [0,1],$ 
 $\mathfrak J_{\xi}= \mathfrak{id}(J_{\xi}).$

 \smallskip
	\item[(ii)]  $\mathfrak J_{0}^{+}=
	\mathfrak{id}(J_{0}^{+})$.

	\smallskip
	\item[(iii)]  $\mathfrak J_{1}^{-}=
	\mathfrak{id}(J_{1}^{-})$.

	\smallskip
	\item[(iv)] For each rational $0<a/b< 1$,
 $\mathfrak J_{a/b}^{\pm}=
\mathfrak{id}( J_{a/b}^{\pm}).$ 
 \hfill$\Box$
\end{enumerate}
\end{corollary}

Corollaries
\ref{corollary:ell-ideals} and 
\ref{corollary:primitive-ideals}, in combination with
Lemma \ref{lemma:lattice-quotient}(i)-(ii) yield:

\begin{corollary} 
{\rm \cite[Theorem 3.2(iii)]{mun-milan}}
\label{corollary:K0-list} 
Letting $\cong$ denote
unital $\ell$-group isomorphism, we have:

\medskip		

\begin{itemize}
\item[(i)] $K_{0}(\mathfrak M_1/\mathfrak J_{p/q})\cong
(M_1,1)/J_{p/q}\cong(\mathbb Z,q)$,
for every rational $p/q\in \interval$.

\medskip			
\item[(ii)] Let $\mathbb Z \lex \mathbb Z$ denote the
free abelian group 
$\mathbb Z^{2}$
equipped with the {\rm lexicographic ordering}: 
$
(x,y)\geq 0\,\,\,\, \mbox{  iff \,\,\,either }
x>0  \mbox{ or }  (x=0  \mbox{ and }y\geq 0).
$  
Then 
$$K_{0}(\mathfrak M_1/\mathfrak J_{0}^{+})\cong
(M_1,1)/J_{0}^{+}\cong(\mathbb Z \lex \mathbb Z,(1,0)) 
\cong(M_1,1)/J_{1}^{-}
\cong K_{0}(\mathfrak
M_1/\mathfrak J_{1}^{-}).$$

\medskip			
\item[(iii)] For any rational  
$p/q\in \{x\in \mathbb R \mid 0<x<1\}$,  
written in  its lowest terms, let
 $p^{-1}\in \{1,\ldots,q-1\}$
be the inverse of $p$ modulo $q$.
Then
$$K_{0}(\mathfrak M_1/\mathfrak J_{p/q}^{+})\cong
(M_1,1)/J_{p/q}^{+}\cong
(\mathbb Z \lex \mathbb Z,(q,-p^{-1})).$$

\smallskip
\item[(iv)]  
$K_{0}(\mathfrak M_1/\mathfrak J_{p/q}^{-})
\cong (M_1,1)/J_{p/q}^{-}\cong
(\mathbb Z \lex \mathbb Z,(q,p^{-1})).$  

\end{itemize} 
\end{corollary}

\begin{corollary}
{\rm  \cite[Proof of Theorem 3.1]{mun-adv}}
\label{corollary:adjoint}
For every
every irrational $\theta\in\interval$, let 
$\mathfrak F_\theta$ denote the Effros-Shen algebra of
\cite[\S 10]{eff}. Then 

\begin{equation}
\label{equation:inter}
K_{0}(\mathfrak M_1/\mathfrak J_{\theta})\cong
K_{0}(\mathfrak M_1)/K_0(\mathfrak J_{\theta})\cong
(M_1,1)/J_{\theta}\cong(\mathbb Z+\mathbb Z\theta)
\cong K_0(\mathfrak F_\theta).
\end{equation}
\end{corollary}
\begin{proof} 
Combine
\eqref{equation:inversa}  
with the preservation properties of
$K_0$,  
Lemma \ref{lemma:lattice-quotient}(i)-(ii).
\end{proof}

\begin{corollary}
{\rm \cite[Proposition 4]{boc}}\,\,\, 
	\label{corollary:primitive-quotients}
With the notation of
Corollary \ref{corollary:primitive-ideals},
the  primitive quotients of
$\mathfrak M_1$ are as follows:

\begin{itemize}
		
\item[(i)] For each  $\theta\in[0,1]\setminus
 \mathbb Q$, \,\,\,$\mathfrak M_1/\mathfrak J_{\theta}
 =\mathfrak F_{\theta}$.  

\medskip
\item[(ii)] $\mathfrak M_1/\mathfrak J_{0}= \mathfrak M_1/\mathfrak J_{1}=$ $\mathbb C$.  For
any rational $0< p/q<1$ (written  in its
lowest terms), $\mathfrak M_1/\mathfrak J_{p/q} = \mathbb M_{q}$.

\medskip
\item[(iii)] For every separable Hilbert space $L$, let $\mathbb K(L)$
denote the C$^{*}$-algebra of all compact operators on $L$.  Further,
let $H_{1}$ stand for an infinite-dimensional separable Hilbert space. 
Then $\mathfrak M_1/\mathfrak J^{+}_{0}
=\mathfrak M_1/\mathfrak J^{-}_{1}$ is the
Behnke-Leptin (unital) AF algebra $\mathcal A_{0,1}$ 
of operators on the Hilbert space $H=\mathbb
C \oplus H_{1}\otimes \mathbb C$ generated by $\mathbb K(H)$ and
$1_{H_{1}}\otimes \mathbb K(\mathbb C),\,$  $($see
\cite[p.330-331]{behlep}$)$.

\medskip
\item[(iv)] For any rational  $0<p/q<1$,
 letting $k\equiv -p^{-1}\mod q,
\,\,\,k \in \{1,\ldots,q-1\}$, \,\,\,
 $ \mathfrak M_1/\mathfrak J^{+}_{p/q}$ is the %
Behnke-Leptin AF algebra $\mathcal A_{k,q}$ of operators on the
Hilbert space $H=\mathbb C^{k}\oplus H_{1}\otimes \mathbb C^{q}$
generated by $\mathbb K(H)$ and $1_{H_{1}}\otimes \mathbb K(\mathbb
C^{q})$.

\medskip
\item[(v)] 
For any rational  $0<p/q<1$, letting 
 $h \equiv p^{-1}\mod q, \,\,\, h\in
\{1,\ldots,q-1\}$, \,\,\,
$ \mathfrak M_1/\mathfrak J^{-}_{p/q}$ is the Behnke-Leptin
  AF algebra $\mathcal A_{h,q}$.

 \end{itemize}
	\end{corollary}	
	
\medskip	 
By  Theorem  \ref{theorem:generator}, 
 the involutive monoid
 $E(\mathfrak M_1)$ has two possible
 generators.
Let us choose,
 once and for all,  a projection
 ${\mathsf p} \in \mathfrak M_1$ such that 
 $[\mathsf p]$ generates $E(\mathfrak M_1).$
Then
 $[\mathsf p]$  {\it  freely}  generates $E(\mathfrak M_1).$  
Since the McNaughton function   $\ide$ is a free generator of 
 $\McNone\cong E(\mathfrak M_1)$,  there is
 precisely
 one    isomorphism
$\omega$ of   $ \McNone $ onto   $E(\mathfrak M_1)$
mapping   $\omega(\ide)$ into $[\mathsf p].$ In symbols,
\begin{equation}
\label{equation:omega}
\omega\colon f\in\McNone\cong \omega(f)\in
E(\mathfrak M_1),\,\,\,\, \omega(\ide)=[\mathsf p].
\end{equation}
 Every  ideal   of $E(\mathfrak M_1)$
 is the $\omega$-image of a uniquely determined
 ideal $\mathfrak i$ of $\McNone$ in such a way that
 we have the isomorphism
\begin{equation}
\label{equation:omega-quotient}
[q]/\omega(\mathfrak i)\in E(\mathfrak M_1)/\omega(\mathfrak i) 
\cong
\omega^{-1}([q])/\mathfrak i\in \McNone/\mathfrak i,
\,\,\,\mbox{ for all } [q]\in E(\mathfrak M_1).  
\end{equation}

\begin{proposition}
\label{proposition:main-generator}
As in Theorem \ref{theorem:robinson-apal}(i),
 for  any irrational $\theta$  in $\interval,$
let  the  ideal $\mathfrak j_\theta$ of $\McNone$ be 
defined by 
  $\mathfrak j_{\theta}=\{f\in  \McNone
 \mid f(\theta)=0\}.$

\smallskip
(i)
There is
a unique embedding $\eta_\theta$ of  the 
quotient 
 $\McNone/\mathfrak j_\theta$  
into  the
 standard  MV algebra $\interval$ 
 introduced in the 
 proof of Theorem \ref{theorem:generator}(b).  

\smallskip
(ii) There is
a unique embedding $\iota_\theta$ of 
  $E(\mathfrak M_1)/\omega(\mathfrak j_\theta)$ 
into  $\interval.$  

\smallskip
(iii)
The projection  ${\mathsf p}\in \mathfrak M_1$ 
has the  following property:  For each irrational $\theta\in \interval$,
 \,\,$\iota_\theta([\mathsf p]/\omega(\mathfrak j_\theta))=\theta.$
\end{proposition}

\begin{proof} (i) 
Since $\mathfrak j_\theta$ is a maximal
ideal of  $\McNone$ then 
 $\McNone/\mathfrak j_\theta$
  is {\it simple}, i.e., $\{0\}$ is its only ideal.   
%
By   \cite[Theorem 4.16(i)]{mun11},
there is  precisely one  
embedding $\eta_\theta$ of   
 $\McNone/\mathfrak j_\theta$  into   $\interval$. 
In more detail, let $\gen(\theta)$ denote the
MV  subalgebra of   $\interval $   
generated by $\theta.$ Then $\eta_\theta$
is an isomorphism of  $\McNone/\mathfrak j_\theta$ onto
$\gen(\theta).$ 

\smallskip
(ii) 
%
 %
 %
By \eqref{equation:omega-quotient},
$E(\mathfrak M_1)/\omega(\mathfrak j_\theta)$
is isomorphic to 
$ \McNone/\mathfrak j_\theta.$
%
Now apply  (i).

\medskip
(iii) By \eqref{equation:omega}-\eqref{equation:omega-quotient},
it is enough to prove that 
the unique embedding
$$
\eta_\theta\colon  f /\mathfrak j_\theta \in 
\McNone/\mathfrak j_\theta \mapsto 
\eta_\theta(f/\mathfrak j_\theta)
\in \interval 
$$
satisfies
 $\eta_\theta(\ide/\mathfrak j_\theta)=\theta.$
 The proof of (i) allows us to identify 
the MV-algebras  $\gen(\theta)$ and 
$\McNone/\mathfrak j_\theta$. It is well known
that the only automorphism of  the MV algebra
$\gen(\theta)$ 
is identity, (see \cite[Theorem 4.16]{mun11}(i), or 
apply  
H\"older's theorem, \cite{bigkeiwol}). 
Two McNaughton functions
 $f,g \in \McNone$ satisfy  $f/ \mathfrak j_\theta=
g/ \mathfrak j_\theta$ iff they have the same value at $\theta$.
Thus the embedding 
$\eta_\theta$ acts on  
$f/ \mathfrak j_\theta$ as evaluation of $f$
at $\theta$, in symbols,
$
\eta_\theta(f/ \mathfrak j_\theta)=f(\theta) \mbox{ for all
} \,\, f\in \McNone.
$
In particular, 
 $
 \eta_\theta(\ide/\mathfrak j_\theta)=\ide(\theta)=\theta,$
 as desired.
\end{proof}

\section{The word problem of the
Farey algebra  $\mathfrak M_1$}
In view of Theorem
\ref{theorem:multi},
the word problem of  the   
 AF$\ell$ algebra
 $\mathfrak M_1$
can   be defined mimicking   the classical definition 
of the word problem for a group presented  by 
 generators and relations, 
 \cite[\S VIII, 1.7]{man}. 
 Thus,  given two formulas  $\phi$ and $\psi$ 
 in the language of involutive
 monoids, the word problem asks to 
 mechanically decide whether
 $\phi$ and $\psi$   code the same 
  equivalence of projections of $\mathfrak M_1$. 
 
 In the next section we will similarly define the
 word problem of every
 AF$\ell$-algebra  $\mathfrak B$ whose Elliott
 involutive monoid is singly generated. 
 We will be able to do so because
  the freeness
 properties of $E(\mathfrak M_1)$ ensure that
 $E(\mathfrak B)$ is a quotient of $E(\mathfrak M_1)$.
 The main aim of this paper is to assess  the
 algorithmic complexity
the word problem of various quotients of $\mathfrak M_1$.

\subsubsection*{Syntax/Semantics of the word problem.}
To get to work on  (Turing)  computational issues
like word problems, a modicum of
syntactical machinery must be introduced. To this purpose,
first of all, we 
 enrich  the language  (= set of constant and
 operation symbols)   $0,^*,\oplus$ of 
 one-generator involutive
monoids by adding a new constant symbol $X$ for the generator.
Throughout, by a {\it formula} we mean a string of symbols
over the alphabet $\{0,X,^*,\oplus,),(\}$, given by the following inductive
definition: 0 and $X$ are formulas; if  $\phi$ and $\psi$ are formulas,
then so are $\phi^*$ and $(\phi\oplus\psi).$ It is a tedious
exercise to verify
 that   this definition ensures ``unique readability'' of every
formula $\phi$, i.e., its mechanical unique decomposability
into  immediate subformulas, unless $\phi$ is a constant
symbol.
 
Every formula  $\phi$ in the language
 $0,X,^*,\oplus$\,\,\,   codes precisely one
McNaughton function ${\mathsf f}_\phi$ 
by the following inductive definition on the number of
 operation symbols in subformulas of $\phi$: \,\,\,
  ${\mathsf f}_0$ is 
  the zero function on $\interval$, \, and \,  ${\mathsf f}_X$  is
   $\ide.$  Inductively,  ${\mathsf f}_{\psi^*}=
   1-{\mathsf f}_{\psi}={\mathsf f}^*_\psi$,
  and ${\mathsf f}_{(\psi\oplus \chi)} = \min(1, {\mathsf f}_\psi+{\mathsf f}_\chi)={\mathsf f}_\psi\oplus {\mathsf f}_\chi.$
This  definition makes sense because of the
unique readability of formulas.
  We say that ${\mathsf f}_\psi$ is the {\it McNaughton function of}
 $\psi$.
 
   Figure \ref{figure:mcn}  shows  the McNaughton function of  
 the formula  
$$
((((X\oplus X)^*
 \oplus X)\oplus(X\oplus X)^*)^*\oplus(X\oplus X)^*)^*.
$$

Via the isomorphism $\omega$ of \eqref{equation:omega}, 
every formula 
$\psi$ simultaneously  codes   the 
  Murray-von Neumann
 equivalence class of projections  $\omega({\mathsf f}_\psi)$
of  $\mathfrak M_1.$

 \begin{figure} 
    \begin{center}                                     
    \includegraphics[height=7.9cm]{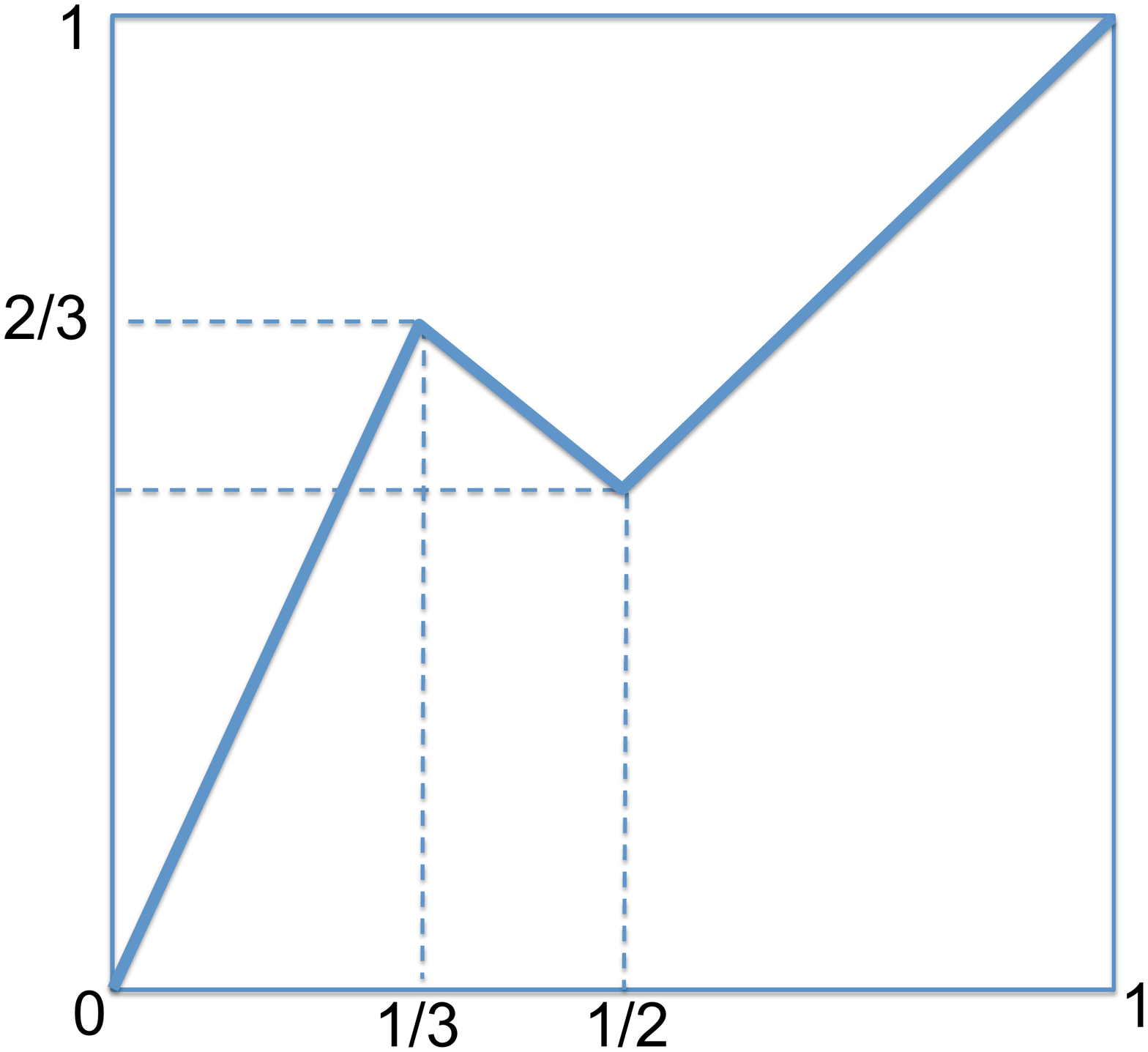}    
    \end{center} 
     \caption{\small The  $\omega^{-1}$-image 
     of the Murray-von Neumann
 equivalence class   
 $[r]=(((([\mathsf p]\oplus [\mathsf p])^*
 \oplus [\mathsf p])\oplus([\mathsf p]\oplus [\mathsf p])^*)^*\oplus([\mathsf p]\oplus [\mathsf p])^*)^*\in E(\mathfrak M_1)$. 
}                                      
 \label{figure:mcn}                                                 
   \end{figure}

 \medskip

\begin{definition}
\label{definition:word}
The  {\it word problem} of 
 $\mathfrak M_1$ is as follows:
 
 \smallskip
 \noindent
 ${\mathsf{INSTANCE}:}$ Formulas  $\phi$ and $\psi$ in the
 language of $0,X,^*,\oplus$.
 
  \smallskip
 \noindent
 ${\mathsf{QUESTION:}}$  Are   $\omega({\mathsf f}_\phi)$ and
 $\omega({\mathsf f}_\psi)$ the same equivalence class of
 projections of $\mathfrak M_1$?
 \end{definition}

 Should the symbol $X$ be interpreted as
 the other generator of   $E(\mathfrak M_1)$, 
 Theorem \ref{theorem:generator} ensures that
 every question in the resulting
 word problem has precisely the same answer
 as in the original word problem defined above.
 
 \medskip
 For any formula  $\phi$
in the language
 $0,X,^*,\oplus$ 
 we let
$ ||\phi|| =$
  the number of occurrences of symbols in  
$ \phi.$
 
\smallskip
 \begin{theorem}
 \label{theorem:polytime-emmone}
   The word problem of 
 $\mathfrak M_1$ is solvable in polynomial time.
In other words, there is
  a polynomial $T\colon\{0,1,2,\dots\}\to\{0,1,2,\dots\}$ 
  and a Turing machine  which, having in input
 any pair of formulas $(\phi,\psi)$,  decides
 whether they code the same equivalence class of
 projections of $\mathfrak M_1$, 
in  a number of steps
 $< T( ||\phi||+ ||\psi||).$ 
 \end{theorem}
 
 \begin{proof} In view of the isomorphism 
 \eqref{equation:omega},  we
 may replace  $E(\mathfrak M_1)$ by
 $\McNone$ and focus on
 the complexity of  the equivalent  problem
 of checking 
 whether the McNaughton functions ${\mathsf f}_\phi$ 
 and ${\mathsf f}_\psi$
 coincide.
We first make the following:

 \smallskip
 \noindent
 {\it Claim 1:} Let $\rho$ be a formula.
For  all $0<z<1$ both  one-sided derivatives 
$\partial {\mathsf f}_\rho(z)/\partial x^\pm$  
  (exist and) satisfy
$
|\partial {\mathsf f}_\rho(z)/\partial x^\pm|\leq 
||\rho||.
$

\smallskip
Existence immediately follows
from  piecewise linearity.
Since every linear piece of every
  $g\in \McNone$  has integer coefficients, 
   both  one-sided derivatives of $g$ at $z$ 
 have integer values.  For the formulas  $0$ and $X$
 the claim is trivially true.
Proceeding inductively on the number of
operation symbols in  subformulas  $\sigma$ of $\rho$,
we can write:
$$
\left|\frac{\partial {\mathsf f}_{\sigma^*}(z)}{\partial x^+}\right|
=
\left|\frac{\partial {\mathsf f}_{\sigma}(z)}{\partial x^+}\right|
\leq ||\sigma|| < ||\sigma^*||\leq ||\rho||. 
$$
On the other hand, whenever $\sigma$ has the form
$(\alpha\oplus\beta)$, 
we can write:

\begin{equation}
\label{equation:cases}
\frac{\partial {\mathsf f}_\sigma(z)}{\partial x^+}=
\frac{\partial {\mathsf f}_\alpha(z)}{\partial x^+}  +
\frac{\partial {\mathsf f}_\beta(z)}{\partial x^+},
\end{equation}
unless   ${\mathsf f}_\alpha(z)\oplus{\mathsf f}_\beta(z)=1$,
 in which case
\begin{equation}
\label{equation:cases-bis}
\frac{\partial {\mathsf f}_\sigma(z)}{\partial x^+}=
\min\left(0,\,\,\,
\frac{\partial {\mathsf f}_\alpha(z)}{\partial x^+}  +
\frac{\partial {\mathsf f}_\beta(z)}{\partial x^+}\right).
\end{equation}
In conclusion, 
$$
\left|\frac{\partial {\mathsf f}_{\sigma}(z)}
{\partial x^+}\right|
\leq 
\left|\frac{\partial {\mathsf f}_{\alpha}(z)}{\partial x^+}
+
\frac{\partial {\mathsf f}_{\beta}(z)}{\partial x^+}\right|
\leq 
\left|\frac{\partial {\mathsf f}_{\alpha}(z)}{\partial x^+}\right|
+
\left|\frac{\partial {\mathsf f}_{\beta}(z)}{\partial x^+}\right|
\leq ||\alpha||+||\beta||< ||\sigma||\leq ||\rho||.
$$
A similar result holds for 
$\partial {\mathsf f}_\sigma(z)/\partial x^-$.
Our  first claim is settled.

\bigskip
 \noindent
 {\it Claim 2:}
 Let  
$\phi$ and $\psi$ be formulas. If  ${\mathsf f}_\phi\not= {\mathsf f}_\psi$, there is
 a rational $r =c/d$ such that ${\mathsf f}_\phi(r)
 \not= {\mathsf f}_\psi(r)$ and   
$0 < d \leq 14(||\phi||+||\psi||).$

\smallskip
As a matter of fact, let $h=|{\mathsf f}_\phi-{\mathsf f}_\psi|.$  
Then a tedious but straightforward
calculation (\cite[1.2.4]{cigdotmun}) yields
\begin{equation}
\label{equation:distance}
h  = ({\mathsf f}_\phi^*\oplus {\mathsf f}_\psi)^*  \oplus ({\mathsf f}_\psi^*\oplus {\mathsf f}_\phi)^*=
{\mathsf f}_{((\phi^*\oplus \psi)^*  \oplus (\psi^*\oplus \phi)^*)}
\end{equation}
showing that  $h$  is a member of $\McNone.$
If $h$ is the constant 1 we are done.
Otherwise,  the maximum value of $h$ 
is  attained at some  point 
$r \in \interval$ of $h$  which is
the unique solution of
a linear equation
$d'x+c'=d''x+c''$ with integer coefficients $c',c'',d',d''.$
(If  $r\in \{0,1\}$ the proof is trivial.)
The denominator $d$
of $r$ satisfies  $0< d\leq |d'-d''|\leq |d'|+|d''|.$ 
By definition, 
$|d'|, |d''|\leq |\partial h(r)/\partial x^+|.$
By Claim 1, 
$$
d \leq |d'|+|d''| \leq 
2\left|
\frac{\partial h(r)}{\partial x^+}\right|  \leq  2||((\phi^*\oplus\psi)^*
\oplus (\psi^*\oplus\phi)^*)||. 
$$
Now 
$$
2||(\phi^*\oplus\psi)^*
\oplus (\psi^*\oplus\phi)^*||=2(2(||\phi||+||\psi||)+10)
\leq2(2(||\phi||+||\psi||)+5((||\phi||+||\psi||)) 
$$
which settles our second claim.

 \smallskip
To conclude,  it suffices to prove:

\bigskip
 \noindent
 {\it Claim 3:}
The problem whether
for some rational 
$r \in \interval$ with
denominator  
$0< d \leq 14 (||\phi||+||\psi||)$
the function $h$ vanishes at $r$,
has polytime complexity with respect to 
$||\phi||+||\psi||$.
%
%
 
 \smallskip
As a matter of fact, by 
  Claim 2, the  number $N$ of
these rational numbers $r\in \interval$  
satisfies $N< (14  (||\phi||+||\psi||))^2$.
By \eqref{equation:distance},
for any such $r$, 
we must compute the value 
${\mathsf f}_\sigma(r)$ for all  
 subformulas  $\sigma$
of $((\phi^*\oplus \psi)^*  \oplus (\psi^*\oplus \phi)^*)$.
This is done  proceeding by induction
on the number of operation symbols in $\sigma$.
The number of these subformulas is bounded above by
a linear function of $||\phi||+||\psi||.$
Since both  $^*$ and $\oplus$
are polytime operations, it easily follows that    
computing the value ${\mathsf f}_\sigma(r)$
takes polynomial time in $||\phi||+||\psi||,$
whence so does the computation of $h(r).$

\smallskip
Having thus settled Claim 3, we have completed
the  proof that
the word problem
of $\mathfrak M_1$ has polytime complexity.
 \end{proof}

 \section{The word problem of Behnke-Leptin primitive
 quotients of $\mathfrak M_1$}
The freeness properties of the
Farey algebra $\mathfrak M_1$  
yield a natural definition of the word problem
of every  AF$\ell$ algebra   $\mathfrak Q$   
whose Elliott involutive monoid
  $E(\mathfrak Q)$ is singly generated.
Since
  $E(\mathfrak Q)$  is the quotient of  the free algebra
$E(\mathfrak M_1)$ by some ideal  
  \,\,$\mathfrak i$ \,\,of $E(\mathfrak M_1)$, then
by  Lemma \ref{lemma:lattice-quotient},
  $\mathfrak Q$ is  
the quotient of $\mathfrak M_1$ by the
ideal $\mathfrak I$ corresponding to $\mathfrak i$
via the functors   $K_0$
and  $\Gamma$. 
    Once   $E(\mathfrak M_1)$ is equipped with the
  generator $[\mathsf p]$ introduced in Proposition 
\ref{proposition:main-generator}(iii), 
   $E(\mathfrak Q)$ is automatically endowed with the
   generator $ [\mathsf p/\mathfrak I]\in E(\mathfrak M_1/\mathfrak I)$.

\smallskip 
 The {\it word problem} of $\mathfrak Q$ asks
whether  two arbitrary formulas
$\phi,\psi$ in the language
$0,X,^*,\oplus$  code  the same equivalence class of
projections 
in  $\mathfrak Q$, 
with  the symbol $X$ coding 
the generator 
$[\mathsf p/\mathfrak I]$  of    the involutive
monoid  
$E(\mathfrak M_1/\mathfrak I)\cong
E(\mathfrak M_1)/\mathfrak i=
E(\mathfrak Q)$.
When $\mathfrak Q=\mathfrak M_1$ we recover
Definition \ref{definition:word}.

\begin{figure} 
    \begin{center}                                     
    \includegraphics[height=8.9cm]{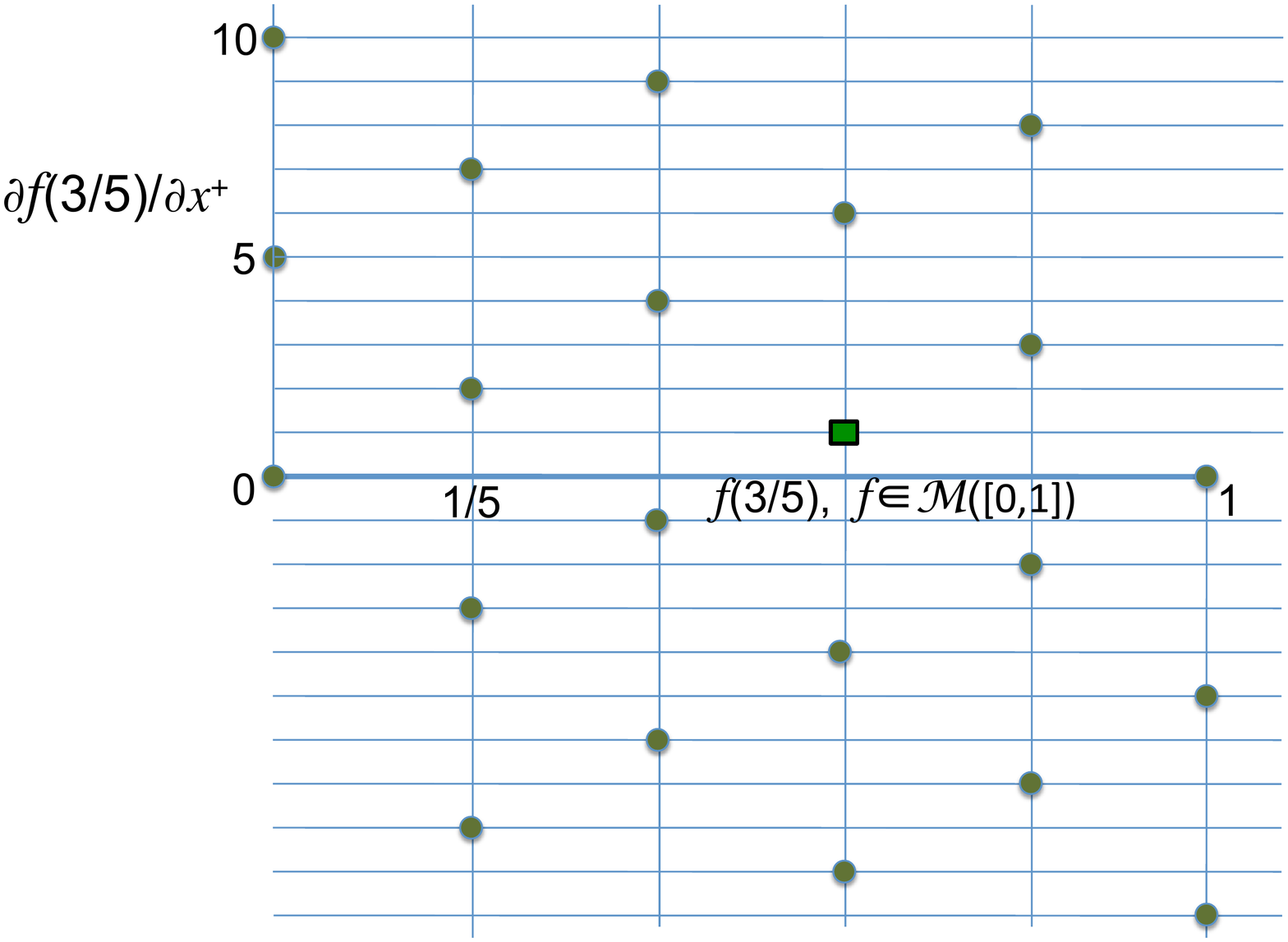}    
    \end{center}                                       
 \caption{\small Elements   of the
  Elliott  involutive monoid  
  of the Behnke-Leptin algebra   
 $\mathcal A_{3,5}= \mathfrak M_1/\mathfrak J^+_{3/5}.$
 Each point stands for the germ
  at $3/5$  of some  function
$f\in \McNone$.
The abscissa plots  the possible values $f(3/5)$,
as  $f$ ranges over all
 McNaughton functions.
Correspondingly, the ordinate  plots
the  possible values of $\partial f(3/5)/\partial x^+$.
   The  point  $(3/5,1)$
is  the generator of 
  $E(\mathcal A_{3,5})$
corresponding to  $\ide/\mathfrak j^+_{3/5}$.
With reference to  Corollary  \ref{corollary:K0-list}(ii)-(iii),
the map
$(x,y)\mapsto (5x, (y-10x)/5)$ is an isomorphism
of   $E(\mathcal A_{3,5})$  onto  
$\Gamma(\mathbb Z \lex \mathbb Z,(5,-2))$.}
    \label{figure:bl}                                                 
   \end{figure}

\smallskip
In \cite[pp.36-37]{ell-lnm} Elliott constructs the local
semigroup of the Behnke-Leptin (unital) separable
AF algebra
$\mathcal A_{0,1}$   with a two point dual,  \cite{behlep}.  Since
$\mathcal A_{0,1}$  has comparability in the sense of
Murray-von Neumann, 
$\mathcal A_{0,1}$   is an AF$\ell$ algebra.
Further, 
Elliott's analysis in \cite{ell-lnm} shows that the involutive
monoid
$E(\mathcal A_{0,1})$   is singly generated.
We then have:

 \begin{theorem}
 \label{theorem:polytime-bl} 
 The word problem
 of the    Behnke-Leptin C*-algebra  $\mathcal A_{0,1}$ 
 has polynomial time complexity. 
 \end{theorem}
 
 \begin{proof}  
As shown in   \cite[p. 36]{ell-lnm},\,\,\,
  $K_0(\mathcal A_{0,1})$
is the lexicographic product  $\mathbb Z +_{\rm lex} \mathbb Z$ 
of two copies of $\mathbb Z$, and the image
 of the unit element of $\mathcal A_{0,1}$ in $K_0$ is the
order unit $(1,0)$ of $\mathbb Z +_{\rm lex} \mathbb Z$.
By  \eqref{equation:omega}, we may 
  argue for $\McNone$ instead of 
 its isomorphic copy
$E(\mathfrak M_1)$.
By  Theorem \ref{theorem:multi}(v) and
Corollaries
\ref{corollary:ell-ideals}(ii)-\ref{corollary:primitive-quotients}(ii),
 $E(\mathcal A_{0,1})$ 
  is (isomorphic to) the quotient of   
  $\McNone$
 by the ideal $\mathfrak j_0^+$
  of all McNaughton 
  functions that vanish on a right open neighborhood of 0.   
Two elements  $f,g\in \McNone$ are mapped
 into the same element of $E(\mathcal A_{0,1})$
 by the quotient map
  $
 l\in  \McNone \mapsto l/\mathfrak j_0^+\in 
  \McNone/\mathfrak j_0^+$
    iff they have equal value
 in an open right neighborhood of 0.
 By \eqref{equation:omega-quotient}, the
 image
under the quotient map 
 of the generator $[\mathsf p]$  of  
 $E(\mathfrak M_1)$ 
may be  identified with 
 the element ${\rm id}_{\interval}/\mathfrak j_0^+$
 of $  \McNone/\mathfrak j_0^+$
where, as the reader will recall,
${\rm id}_{\interval}$ denotes the identity function
 over $\interval.$  The element 
 ${\rm id}_{\interval}/\mathfrak j_0^+$  
 is known as  the {\it germ} of $\ide$ at 0.
   Every $f\in \McNone$ is similarly transformed 
   by the quotient map 
 into its germ at 0.  Piecewise linearity ensures that
any   $f,g\in \McNone$
 are identified by the quotient map
  iff they have the same germ at 0  
 iff they coincide on some open right neighborhood of 0,
 iff 
 $$
 f(0)=g(0)\,\,\, \mbox{ and }\,\,\,
 \frac{\partial f(0)}{\partial x^+}=
  \frac{\partial  g(0)}{\partial x^+}.
 $$
 (Compare with  the proof of
 Claim 1 in Theorem   \ref{theorem:polytime-emmone}.)

Thus to evaluate  the complexity of the 
 problem  whether 
two formulas  $\phi,\psi$ 
satisfy  ${\mathsf f}_\phi/\mathfrak j_0^+
={\mathsf f}_\psi/\mathfrak j_0^+$ in
$\McNone/\mathfrak j_0^+$, we consider the
equivalent problem of deciding  whether
${\mathsf f}_\phi$ and ${\mathsf f}_\psi$
have the same value at 0
and the same partial derivative at 0 along $x^+.$

We will prove the existence of a polynomial $P$
such that  both problems are decidable in less
than $P(||\phi||+||\psi||)$ steps, for any $\phi$ and $\psi$.
We proceed by induction on the number of
operation symbols in subformulas $\sigma$ of  
  $\phi$ and $\psi$.

  Since 
  ${\mathsf f}_\sigma$ is piecewise linear with integer
  coefficients, then   ${\mathsf f}_\sigma(0)\in \{0,1\}.$
  There are only polynomially
 few subformulas $\sigma$ of $\phi$ and $\psi$.
 Since both  $^*$ and $\oplus$
are polytime operations, it easily follows that    
the computation of the value 
 ${\mathsf f}_\sigma(0)$
 takes polynomial time in $||\phi||+||\psi||$.
Thus the problem whether
 ${\mathsf f}_\phi(0)={\mathsf f}_\psi(0)$  is polytime.
 
By   induction on the number of operation symbols, we
next evaluate the one-sided derivatives
 $$
 \frac{\partial {\mathsf f}_\phi(0)}{\partial x^+}
 \,\,\, \mbox{ and }\,\,\,
  \frac{\partial {\mathsf f}_\psi(0)}{\partial x^+} 
 $$
 as follows:
 First of all,  the derivative of the generator $\ide={\mathsf f}_X$ 
 of $\McNone$  is equal to 1, and the derivative of the zero
 element  ${\mathsf f}_0$  is 0.
 Proceeding inductively, whenever 
 the operation $^*$ is applied to a subformula  $\sigma$
 to yield $\sigma^*$, the one-sided derivative
 flips from  $d$ to $-d.$ Whenever a subformula $\sigma$
 arises as  $(\alpha\oplus \beta)$,  the one-sided
 derivatives of $\sigma$ are
 evaluated arguing by cases and using
 \eqref{equation:cases}-\eqref{equation:cases-bis}
  in the proof
 of Claim 1 in Theorem
  \ref{theorem:polytime-emmone}.
All these derivatives have integer values 
$\leq\max(||\phi||,||\psi||)$, which can
be computed in polynomial time. Thus the computation
of   the one-sided derivatives of ${\mathsf f}_\phi$ and
${\mathsf f}_\psi$ 
at 0
 requires polynomial time  in $||\phi||+||\psi||$.

We have just shown the existence of the desired
 polynomial $P$
such that 
 in less
than $P(||\phi||+||\psi||)$ steps it is decidable if
two formulas  $\phi,\psi$ 
satisfy  ${\mathsf f}_\phi(0)/\mathfrak j_0^+
={\mathsf f}_\psi(0)/\mathfrak j_0^+$.

As a consequence,  the word problem of
  $\mathcal A_{0,1}$ has polytime complexity. 
 \end{proof}
 
\begin{remark}
   $E(\mathcal A_{0,1})$ is
 isomorphic to  the MV  algebra  C  introduced by Chang in 
\cite[p.474]{cha58}. 
\end{remark}

   \begin{corollary}
   \label{corollary:polytime-other-bl}
The word problems of 
the  Behnke-Leptin algebras   
  $\mathfrak M_1/\mathfrak J^+_{m/n}$
  and $\mathfrak M_1/\mathfrak J^-_{m/n}$
have polynomial time complexity. 
(Notation of Corollary   
\ref{corollary:primitive-quotients}.) 
   \end{corollary}
   
   \begin{proof} A routine variant of the
   proof of Theorem \ref{theorem:polytime-bl}.
   For every formula $\psi$,
    one has now to compute the value of ${\mathsf f}_\psi$ and
  both its one-sided derivatives at $m/n$.
   \end{proof}

 Figure \ref{figure:bl} shows  the Elliott involutive
 monoid of a typical Behnke-Leptin algebra 
 $\mathfrak M_1/\mathfrak J^+_{m/n}$.
 
\medskip

 \section{The word problem of Effros-Shen primitive
 quotients of $\mathfrak M_1$}
 
 Following  \cite[p. 65]{eff}, 
  for  each irrational $\theta
 \in [0,1]$ the Effros-Shen  AF algebra $\mathfrak F_\theta$ 
is defined by  
$
K_0(\mathfrak F_\theta) 
\cong (\mathbb Z+\mathbb Z\theta,1),
$
the latter denoting the additive subgroup of $\mathbb R$ 
generated by $\theta$ and 1, with the natural ordering and 
1 as the  unit.

It is well known that Effros-Shen algebras are related to
continued fractions, \cite{bratteli1, bratteli2}, \cite[\S 5]{eff}. 
The elementary theory of diophantine approximation
is also  a key tool for our results on
the word problem of $\mathfrak F_\theta$.
 
\smallskip
We prepare:

 \begin{lemma}
 \label{lemma:minkowski}
 If  $n\geq 2$ is even and  $h/k$ is a rational (with $k>0$)
satisfying  $p_n/q_n<h/k<p_{n+1}/q_{n+1}$ then 
$k>q_{n+1}>q_k.$
\end{lemma}

\begin{proof}
By   \cite[I, \S 1, Theorem 2]{lan}, the
 integer 2$\times$2 matrix whose rows
are given by $(p_n,q_n)$ and $(p_{n+1},q_{n+1})$
is unimodular.
 Passing to homogeneous
integer coordinates in 
 $\mathbb Z^2$,  
let us consider the
integer vector $(h,k)$.
Our hypothesis about
$h/k$ lying in the interior of the interval
$[p_n/q_n, \,\,p_{n+1}/q_{n+1}]$
amounts to saying that  
 $(h,k)$ lies in the interior of the positive span  
 $\{\lambda(p_n,q_n)+\mu 
 (p_{n+1},q_{n+1})\in \mathbb R^2\mid
 0\leq \lambda,\mu\in \mathbb R\}$ of $(p_n,q_n)$ and 
 $(p_{n+1},q_{n+1})$ in $\mathbb R^2$.
 On the other hand,  $(h,k)$
 lies outside the half-open parallelogram
$\mathcal P= \{x\in \mathbb R^2\mid x= \lambda(p_n,q_n)
+\mu (p_{n+1},q_{n+1}), \,\,\,0 \leq \lambda<1, \,\,\,
0 \leq  \mu<1 \}$.  Indeed, 
 $\mathcal P$ does not contain any integer point other
 than the origin, because
  $\{(p_n,q_n),(p_{n+1},q_{n+1})\}$  is  a free generating
  set of the free abelian group  $\mathbb Z^2.$
  
 Thus  $k\geq q_{n+1}+q_k >q_{n+1}>q_k.$
 \end{proof}

From 
$q_n\geq 2q_{n-2}$ we get
$q_{2n} \geq 2^n q_0 =2^n$ and
$q_{2n+1}  \geq 2^n q_1 \geq 2^n$,
whence
$
q_n\geq 2^{{(n-1)}/{2}} \mbox{ for all } n = 2,3,4,\dots.
$
The following alternative 
estimates, however,  will simplify our calculations
in  Theorem \ref{theorem:polytime-ef-cf}, with no loss
in its content.

\begin{lemma}
\label{lemma:chiuen} 
Let  $p_n/q_n$ be the $n$th principal  convergent
of $\theta$.  We then have:

\begin{itemize}
\item[(i)]
$q_0=1\leq q_1$ and
$
q_n >(3/2)^{n-2}
\mbox{ for all } n\geq 2.
$

\smallskip
\item[(ii)]
For all  $m=1,2,\dots$ and  even $n>\log_2 8m$
we have $q_n>2m.$ 
\end{itemize} 
\end{lemma}

\begin{proof} 
(i) 
We have $q_0=1$ by definition. 
Necessarily,   $q_1\geq 1.$
Thus
by   (\cite[I \S 1, Theorem 1]{lan}),
$q_2\geq  q_{1}+q_{0}>1=(3/2)^{2-2}.$
Proceeding inductively,  for all $n\geq 2.$
we have
$
q_n\geq  q_{n-1}+q_{n-2}>
(3/2)^{n-3}+(3/2)^{n-4}>
(3/2)^2\cdot (3/2)^{n-4}=(3/2)^{n-2}, 
$
as desired.

(ii) By 
(i), 
$n>\log_2 8m\,\,\,\Rightarrow\,\,\,
 n>2 + \frac{\log_2 2m}{\log_2 3/2}
\,\,\,\Rightarrow\,\,\, (3/2)^{n-2}> 2m
\,\,\,\Rightarrow\,\,\,
q_n>2m,$
\end{proof}

\medskip
\begin{lemma}
\label{lemma:claim3}
Let $a_0,a_1,\dots$ denote the partial quotients of $\theta.$
For each  $n=2,3,\dots$ let $\bar a_n=
\max(a_0,\dots,a_n).$ Then 
$q_n <  (2\bar a_nq_1)^n.$
\end{lemma}

\begin{proof}
By   \cite[I, \S 1, p.2]{lan},
$q_n=a_nq_{n-1}+q_{n-2} < q_{n-1}(a_n+1)
\leq 2a_nq_{n-1}$. So, in particular, 
$q_2<2a_2q_{1} < (2\bar a_2q_1)^2. $
Inductively, $q_n < 2a_nq_{n-1}<
2a_n(2\bar a_{n-1}q_1)^{n-1}\leq
(2\bar a_nq_1)(2\bar a_{n}q_1)^{n-1}=
(2\bar a_nq_1)^n.$
\end{proof}

\medskip
A brute force counting argument shows that
  uncountably many
Effros-Shen AF algebras do not have a decidable
word problem. On the other hand,  the following
result yields a large class of Effros-Shen algebras
having a polynomial time  word problem:

 \begin{theorem}
 \label{theorem:polytime-ef-cf}
 Let $\theta\in [0,1]\setminus \mathbb Q$.
 Suppose there is a real  $\kappa>0$ 
 such that for  every $n=0,1,\dots$
the   sequence $[a_0,\dots,a_n]$ of partial quotients
 of  $\theta$ is computable (as a finite list of binary integers) 
 in less than $2^{\kappa n}$  steps.\footnote{i.e.,
some Turing machine over input  
$0\dots0,\,\,\, (n \mbox{ zeros}),$ 
 outputs the list
of binary numbers 
 $[a_0,\dots,a_n]$ in a number of steps bounded
 above by some
{\it  exponential}  function of the input length $n$. Numerical inputs
 in unary notation are customary in recursive analysis and
 in discrete complexity theory,
 \cite{ker}.}  Then 
 the word problem of 
 $\mathfrak F_\theta$  has polynomial time complexity.
 \end{theorem}

  \begin{proof} 
By \eqref{equation:omega}, we may argue for
$\McNone$ in place of   $E(\mathfrak M_1)$. 
By Lemma \ref{lemma:lattice-quotient},
 letting  $\gen(\theta)$ denote the
  MV  subalgebra 
of the standard MV-algebra
generated by $\theta,$  we may write
$$
E(\mathfrak F_\theta)=
  \Gamma(\mathbb Z+\mathbb Z\theta,1)=\gen(\theta)=
  \McNone/\mathfrak j_\theta.
$$
By  \cite[Theorem 4.16]{mun11}, 
$\gen(\theta)$ is uniquely isomorphic to  the quotient of    
$\McNone$ 
by the maximal  ideal $\mathfrak j_\theta$ 
of all McNaughton functions of $\McNone$
vanishing at $\theta.$
The irrationality of $\theta$ ensures that 
every  $f\in \McNone$ vanishing at $\theta$ also 
vanishes over an open neighborhood of $\theta.$
This is so because every linear piece of $f$ has
integer coefficients.  
As in the proof of 
Proposition
\ref{proposition:main-generator},
  the quotient map 
$f\in \McNone\mapsto f/\mathfrak j_\theta\in 
\McNone/\mathfrak j_\theta$ 
boils down to evaluating $f$ at $\theta.$

We must prove the existence of a polynomial
$T$ such that for arbitrary formulas
$\phi$ and $\psi$, the problem whether
$\phi$ and $\psi$  code the
same element of 
$E(\mathfrak F_\theta)$ is decidable in less than 
$T(||\phi||+||\psi||)$ steps.

By  \eqref{equation:distance},   the problem  amounts to
checking whether the   function
$|{\mathsf f}_\phi-{\mathsf f}_\psi|\in \McNone$
belongs to $\mathfrak j_\theta$.
In turn, this problem is equivalent to
checking
whether the function $|{\mathsf f}_\phi-{\mathsf f}_\psi|$
vanishes  at $\theta$.
Let  $\delta$ be the formula
\begin{equation}
\label{equation:absolute}
((\phi^*\oplus \psi)^*  \oplus (\psi^*\oplus \phi)^*).
\end{equation}
The same  verification in  the proof of Claim 2 in 
Theorem \ref{theorem:polytime-emmone}
shows that
  ${\mathsf f}_\delta=|{\mathsf f}_\phi-{\mathsf f}_\psi|.$ 
The quotient map  
 sends ${\mathsf f}_\delta\in \McNone$
 into the real number  ${\mathsf f}_\delta(\theta)\in \interval.$
Observe that  ${\mathsf f}_\delta$ belongs to $\mathfrak j_\theta$
iff ${\mathsf f}_\delta(\theta)=0.$ 

\medskip
\noindent
{\it So, to conclude the proof it is sufficient to prove that
there is a polynomial $T'$ such that for every formula $\rho$, 
the problem whether  ${\mathsf f}_\rho(\theta)=0$
can be solved  in less than  $T'(||\rho||)$ steps.}

\medskip
\noindent
Let   the integer $n_\rho$ be  defined by 
\begin{equation}
\label{equation:enchi}
 n_\rho=2\cdot \lceil \log_2 8||\rho|| \rceil =
2\cdot  \mbox{the smallest integer }\geq  \log_2 8||\rho||.
\end{equation}
By  Lemma  \ref{lemma:chiuen}(ii),
\begin{equation}
\label{equation:final}
q_{n_\rho}>2||\rho||.
\end{equation}

\medskip 
\noindent
{\it Claim 1:}
$
{\mathsf f}_\rho \mbox{ is linear over the interval } I_\rho=[p_{n_\rho}/q_{n_\rho},
p_{n_\rho +1}/q_{n_\rho+1}].
$

\smallskip
Otherwise (absurdum hypothesis),
the piecewise linear function ${\mathsf f}_\rho$ has a nondifferentiability
point $x$ lying in the interior of  
$I_\rho.$  Since the linear pieces of  ${\mathsf f}_\rho$ 
are restrictions of linear polynomials with  integer
coefficients, any such $x$ is the unique solution of an
equation of the form  $d'x+c'=d''x+c''$ for integers
$c',c'',d',d''.$ Thus $x$ is rational,
 and its denominator $d$ 
satisfies the inequalities 
$0<d \leq |d'-d''|\leq |d'|+|d''|$. 
Arguing inductively on the number of operation symbols in
subformulas of 
$\rho$  as
in the proof of  Theorem \ref{theorem:polytime-emmone},
it follows  that the one-sided derivatives $d'$ and $d''$
of ${\mathsf f}_\rho$  at $x$
are bounded above by  the number  $||\rho||$ of symbols of
$\rho.$ Thus  $d\leq 2||\rho||.$ 
{By Lemma \ref{lemma:minkowski}},  the rational
$x$  satisfies   $d >q_{n_\rho+1}>q_{n_\rho}.$
By \eqref{equation:final},
 $d >2||\rho||$, a contradiction that 
 settles our first claim.

 \medskip
Since   
 ${\mathsf f}_\rho$ is piecewise linear  
  and ${\mathsf f}_\rho\geq 0$,  from our claim we have:
    ${\mathsf f}_\rho(\theta)=0$\,\,\, iff \,\,\,$f$ vanishes over
 $I_\rho$\,\,\, iff \,\,\,${\mathsf f}_\rho(p_{n_\rho}/q_{n_\rho})=0
 ={\mathsf f}_\rho(p_{n_\rho+1}/q_{n_\rho+1})$.
Therefore, to   complete the proof of the theorem
there remains to be proved  that the values  
$$v={\mathsf f}_\rho(p_{n_\rho}/q_{n_\rho}),\,\,\,\,
w={\mathsf f}_\rho(p_{n_\rho+1}/q_{n_\rho+1})
$$
 are
 computable in
 polynomial time. 
To this purpose, we first settle the following:

 \medskip
\noindent
{\it Claim 2:}
The time (=number of Turing machine steps)
 $t$ needed to compute the sequence of binary integers
$[a_0,\dots,a_{n_\rho}]$ is bounded above  by
a  polynomial in the length of $\rho.$

\smallskip
As a matter of fact, 
\begin{eqnarray*}
t&<& 2^{\kappa n_\rho},\,\,\,\mbox{ by hypothesis}\\
{}&<& (2^{2(1+\log_2(8||\rho||)})^\kappa,\,\,\, \mbox{ by 
\eqref{equation:enchi}}\\ 
{}&=&4^\kappa\cdot (8||\rho||)^{2\kappa}.
\end{eqnarray*}

 \medskip
\noindent
{\it Claim 3:}
The two sequences
of binary integers 
 integers  $p_0,\dots,p_{n_\rho}$ and
$q_0,\dots,q_{n_\rho}$  
 are computable in polynomial
 time. 
 
 \smallskip
  As a matter of fact, 
 letting 
  $||j||_2$ denote the length of the integer $j$
 written in binary notation,   for some
 constant $c$ and polynomial $P$  we have
 \begin{eqnarray*}
 ||q_{n_\rho}||_2&<&n_\rho\log_2 2q_1\bar a_{n_\rho}, 
 \mbox{ by Lemma }  \ref{lemma:claim3}\\
 {}&<&  (2+2\log_2 8||\rho||)\cdot(c+||\bar a_{n_\rho}||_2),
  \mbox{ by \eqref{equation:enchi}}\\
    {}&<&P(||\rho||),   \mbox{ a fortiori, by  Claim 2} .
 \end{eqnarray*}
Now,   the binary integer 
$q_{n_\rho}$ is computed via  a number $n_\rho$ 
 (logarithmic in $||\rho||$) of applications of
the recursive formulas  \cite[I, \S 1, p.2]{lan},
involving additions and multiplications of the binary integers
  $||a_i||_2,  ||q_i||_2,  ||p_i||_2$     as in the euclidean 
algorithm.  It follows that  the computation
of   
$||q_{n_\rho}||_2$
takes polynomial time (say   $P^3(||\rho||)$),
and so does the computation of $||p_{n_\rho}||_2$.
Our third claim is thus settled.

\smallskip

For every subformula $\sigma$ of $\rho$, 
arguing inductively on the number of operation
symbols of $\sigma$, it follows   that
the computation of  the rational number 
$v={\mathsf f}_\sigma(p_{n_\rho}/q_{n_\rho})$,
as a pair of binary integers, 
requires  polynomial
time in $||\rho||$.  
As a matter of fact,  the number of subformulas
of $\rho$ is bounded above  by  $||\rho||.$
 At each  step,  we have the following two alternatives:
 Either $\sigma$ is transformed into $\sigma^*$,
 whence
 ${\mathsf f}_\sigma^*(p_{n_\rho}/q_{n_\rho})=1-
 {\mathsf f}_\sigma(p_{n_\rho}/q_{n_\rho})$.  Or else, 
 $\sigma=(\alpha \oplus \beta)$, in which case we
 argue by cases using   formulas
 \eqref{equation:cases}-\eqref{equation:cases-bis} as in
  Theorem
\ref{theorem:polytime-emmone}, Claim 1.
Similarly, the value 
$w={\mathsf f}_\rho(p_{n_\rho+1}/q_{n_\rho+1})$ 
is computable in polytime.

We have just shown the existence of 
the desired polynomial $T'$ such that 
 for every formula $\rho$,
 the problem whether  ${\mathsf f}_\rho(\theta)=0$
is decidable in less than  $T'(||\rho||)$ steps. 

As a consequence, 
 the word problem of 
 $\mathfrak F_\theta$  has polynomial time complexity.
\end{proof}

  \begin{corollary}
  \label{corollary:quadratic}
   The word problem
of each  Effros-Shen algebra  $\mathfrak F_{\theta}$,
for  $\theta$   a quadratic
irrational,  or $\theta = 1/e$, 
has polynomial time complexity.
 \end{corollary}
 
 \begin{proof} This follows from a straighforward application of
 the foregoing theorem,  recalling  that both
 $1/e$ and any quadratic irrational satisfy the hypothesis 
 therein. 
 Specifically, for each quadratic irrational one 
 may use  Lagrange theorem,
 \cite[Theorem 3, IV, \S 1, p.57]{lan}.
For Euler's constant  e,  the proof follows by direct inspection of Euler's continued
fraction for e,   \cite[V, \S 2, p.74]{lan}  together with Serret's theorem
 \cite[Theorem 8, I, \S 3, p.14]{lan}.
 \end{proof}
 \noindent
 Fix an irrational $\theta$
and  consider the following decision problem:
 
\medskip
 \noindent
 ${\mathsf{INSTANCE}:}$  Integers  $q>0$
 and $p$,  written
 in {\it binary notation}.
 
  \smallskip
 \noindent
 ${\mathsf{QUESTION:}}$  Does the inequality    
  $p/q<\theta $ hold ?
 
\bigskip 
 
We  say
 that   $\theta$ 
 {\it  has an exponential time computable rational  left cut} 
if  the above problem can be
 decided by a Turing machine in a number of steps
 bounded above  by an exponential function of 
 the length $||(p,q)||$  of the input pair  of binary integers $(p,q)$.
 In the literature on recursive analysis, 
  one similarly defines ``dyadic'' left cuts,
    \cite[Definition 2.7 and Remark, p.48]{ker}.

 \begin{corollary}
 \label{corollary:polytime-ef-cut}
For any irrational $\theta\in \interval$
having an {\em exponential} time computable left cut,
the Effros-Shen  algebra
 $\mathfrak F_\theta$ has a {\em polynomial}
  time word problem. 
 \end{corollary}

 \begin{proof}
Two formulas  $\phi$ and $ \psi$
 code the same element of $\mathfrak F_\theta$ 
iff  the identity
 ${\mathsf f}_\psi(\theta)={\mathsf f}_\phi(\theta)$ holds
in $\McNone$   iff
${\mathsf f}_\delta(\theta)=0$, with 
$\delta$ the formula
$((\phi^*\oplus \psi)^*  \oplus (\psi^*\oplus \phi)^*)$.
Indeed, as the reader will recall, 
 ${\mathsf f}_\delta=|{\mathsf f}_\phi-{\mathsf f}_\psi|.$
It is sufficient to show the existence of a polynomial
$T$ such  that for any 
formula $\rho,$
 the problem of checking whether
  ${\mathsf f}_\rho(\theta)=0$ can be solved
  in less than $T(||\rho||)$ Turing steps.

To this purpose, let  
$$
0=r_0<r_1<\dots< r_k=1, \,\,\,r_i=n_i/d_i\in \interval \cap\, \mathbb Q,\,\,
0<d_i\leq 2 ||\rho||, (i=0,\dots,k)
$$
be the  Farey sequence  of all
 rational points $r_i\in\interval$ of denominator  $\leq 2 ||\rho||$.
 
Since every  rational  $x$ with $r_i< x < r_{i+1}$
has a denominator $d > 2 ||\rho||,$ we claim that
\begin{equation}
\label{equation:linear}
\mbox{${\mathsf f}_\rho$ is linear over each interval 
$[r_i, r_{i+1}]$.}
\end{equation}
For otherwise (absurdum hypothesis),
${\mathsf f}_\rho$ has a nondifferentiability
point $x$ lying in the interior of 
$[r_i, r_{i+1}]$.
Any such $x$ is the unique solution of an
equation of the form  $d'x+c'=d''x+c''$ for suitable integers
$c',c'',d',d''.$ Thus $x$ is rational, 
  and its denominator $d$ 
satisfies the inequalities 
$0<d\leq |d'-d''|\leq |d'|+|d''|$.
As
in the proof of  Theorem \ref{theorem:polytime-emmone}
(Claim 1),
 the values  $d'$ and $d''$ of the one-sided derivatives
 of ${\mathsf f}_\rho$ at $x$ 
are bounded above by   $||\rho||$.
Thus  $d \leq 2||\rho||,$ a contradiction.

\smallskip
There exists a Turing machine   $\mathcal T$
and a (cubic)  polynomial $T'$ with the following properties:
Over any  input  $\rho$, 
$\mathcal T$
 outputs the sequence of intervals $[r_i, r_{i+1}]$ in 
 less than $T'(||\rho||)$ steps.  
To this purpose,  $\mathcal T$
first computes the Farey mediant  $1/2$ of
 0 and 1, and then, inductively,   $\mathcal T$
  inserts
all possible Farey mediants of denominator
$3,4,\dots, \,2||\rho||-1,\,\, 2||\rho||$,  
(as pairs of integers in binary notation)
until the desired Farey
sequence is completed.
 
For each $i=1,\dots, k-1$,  let
 $L_i$ be the length of the quadruple
of binary integers $n_i,d_i,n_{i+1}, d_{i+1}$ denoting the  rational interval
$[r_i, r_{i+1}]$. Then for some constant  $\lambda>0$,\,\,
we can write 
 $L_i< \lambda \log_2 ||\rho||.$
Arguing inductively on subformulas
of $\rho$ as in the proof of Theorem \ref{theorem:polytime-ef-cf},
we may assume that for some polynomial $T''$,\,\,
machine $\mathcal T$  outputs the
list of  (automatically rational)
values ${\mathsf f}_\rho(r_i), \,\,\, (i=0,\dots,k)$
in less than $T''(||\rho||)$ steps.
Each  ${\mathsf f}_\rho(r_i)$ is a pair
of binary integers. 

Since $\theta$
has an exponential time computable rational  left cut,  
$\mathcal T$ detects  
the only interval $[r_{i^*}, r_{{i^*+1}}]$ with 
$\theta\in [r_{i^*}, r_{{i^*+1}}]$
 in exponential time
with respect to $L_i$, and hence,  in polynomial time  
with respect to $||\rho||$.

We have thus shown the existence of the desired
polynomial  $T$, \,\, having the property that 
$\mathcal T$  checks
  whether 
${\mathsf f}_\rho(r_{i^*})={\mathsf f}_\rho( r_{{i+1}^*})=0$
in less than  $T(||\rho||)$ steps. 
By \eqref{equation:linear},
this condition is equivalent to
 ${\mathsf f}_\rho(\theta)=0$.

Thus  the word problem of 
 $\mathfrak F_\theta$ has polytime complexity.
 \end{proof}

 The following result strengthens the first statement of
Corollary   \ref{corollary:quadratic}:
 
 \begin{corollary}
 \label{corollary:collins}
 Let $\theta\in \interval\setminus \mathbb Q$
 be a real algebraic integer.
Then the word problem of 
 $\mathfrak F_\theta$ has polytime complexity.
 \end{corollary}
 
 \begin{proof}  There exists a one-variable polynomial
 $l \colon \mathbb R\to \mathbb R$ with integer coefficients, together with a rational
 interval $[r',r'']\subseteq \interval$ such that $\theta$ is
 the only root of $l$ lying in $[r',r'']$.   
We say that  $(l,[r',r''])$ is   the
 {\it interval  representation}   of $\theta$,  \cite[p.327]{mis}.
Without loss of generality, we may also assume 
\begin{equation}
\label{equation:derivative}
\frac{{\rm d}l(\theta)}{{\rm d}x} > 0.
\end{equation}
 {\it There exists} a Turing
 machine $\mathcal T$ and a constant  $\kappa>0$ such
 for any input pair  $(p,q)$ of binary integers,
 $\mathcal T$ decides whether  $p/q<\theta$ in less than
 $2^{(\kappa ||(p,q)||)}$ steps:  in view of
 \eqref{equation:derivative}, this amounts to checking whether
 $p/q < r'$ or else  $p/q\in [r',r'']$ and  $l(p/q)<0$.
  (Since the data needed for the representation
 $(l,[r',r''])$  are assumed to be given at the outset, independently
 of  $p/q$, a closer inspection shows
 that $\mathcal T$ actually works in polytime.)
Since  $\theta$ has an exponential  time
 rational left cut, the desired conclusion follows from 
 Corollary  \ref{corollary:polytime-ef-cut}.
 \end{proof}

\section{Quotients of 
$\mathfrak M_1$ and G\"odel incomplete word problems}

Differently from the earlier sections, in this
section  familiarity is required with first order logic  (G\"odel
completeness and 
incompleteness, \cite{man, mon}) and \luk\ 
infinite-valued sentential logic 
(decidability,  Lindenbaum algebras,
\cite{cigdotmun, mun11}).  

As we have seen,
many primitive quotients of 
 the Farey AF algebra  $\mathfrak M_1$
 existing in the literature have a polytime word problem.
On the other hand, since there are only countably many
Turing machines,  already
 uncountably many  Effros-Shen quotients 
 $\mathfrak F_\theta$ have
 an undecidable word problem.

An  interesting intermediate class of structures
is given by  quotients $\mathfrak B=\mathfrak M_1/\mathfrak I$ 
 having a  {\it G\"odel incomplete word problem.}
 This means that  the set of pairs of formulas
$(\phi,\psi)$ that code the same equivalence class of
projections of $\mathfrak B$  is recursively enumerable 
and undecidable.

%
 %

\smallskip
The following result shows that 
primitive quotients of $\mathfrak M_1$
are immune to G\"odel incompleteness.

\begin{theorem}
\label{theorem:apal-immune}
Let $\mathfrak P$ be a primitive ideal of $\mathfrak M_1.$
\begin{itemize}

\smallskip
\item[(a)]
If  $\mathfrak P$ is a maximal ideal of 
$\mathfrak M_1$ such that 
$\mathfrak M_1/\mathfrak P$ is finite-dimensional,
then
the word problem of 
$\mathfrak M_1/\mathfrak P$  is solvable in polynomial time.  

\smallskip
\item[(b)]
If  $\mathfrak P$ is a maximal ideal of 
$\mathfrak M_1$
such that 
$\mathfrak M_1/\mathfrak P$ is not finite-dimensional,
and the word problem of 
$\mathfrak M_1/\mathfrak P$ is  recursively enumerable,
 then it  is decidable. 

\smallskip
\item[(c)]
If  $\mathfrak P$ is not a maximal ideal of 
$\mathfrak M_1$ 
the word problem of $\mathfrak M_1/\mathfrak P$
 is solvable in polynomial time.

\end{itemize}
\end{theorem}

\begin{proof} 
We will work with
 $\McNone$ in place of its isomorphic copy $E(\mathfrak M_1)$.
 
 \smallskip
 (a)
 By Corollaries  \ref{corollary:K0-list}(i) and
\ref{corollary:primitive-quotients}(ii),
 $\mathfrak P=\mathfrak J_{p/q}$ for some
 rational $p/q\in \interval$, and  
$E(\mathfrak M_1/\mathfrak J_{p/q})$ is finite. 
Then, trivially,   the word problem of   
$\mathfrak M_1/\mathfrak J_{p/q}$
 is solvable in polynomial time.

\medskip
(b)-(c)
 By Lemma
 \ref{lemma:lattice-quotient}(iii),
the map  
\begin{equation}
\label{equation:latticeiso}
\Gamma\circ K_0\colon 
\mathfrak I\in
 \{\mbox{ideals of } \mathfrak M_1\}\mapsto \mathfrak i
 \in  \{\mbox{ideals of } \McNone\}
 \end{equation}
is a lattice isomorphism.
Since  $\mathfrak P$ is primitive, 
$\mathfrak p$ is {\it prime}
  (meaning that $\McNone/\mathfrak p$
is totally ordered).  
Recalling the one-one correspondence $\xi \in \interval
\mapsto \mathfrak j_\xi$
in Theorem  \ref{theorem:robinson-apal},
  let  $\theta$ be the only point of $\interval$ 
  such that the maximal ideal 
  $ \mathfrak j_\theta$ contains $ \mathfrak p.$
By Theorem  \ref{theorem:robinson-apal},
at most two ideals contain  $\mathfrak p$, namely
 $\mathfrak  j_\theta=\{f\in \McNone\mid f(\theta)=0\}$
and $\mathfrak p$ itself.

\bigskip
\noindent
\emph{Case 1.} 
$\mathfrak p=\mathfrak j_\theta$. 

In view of the   lattice isomorphism
 $\mathfrak I\mapsto \mathfrak i$\,\,\, 
 in \eqref{equation:latticeiso},
 $\mathfrak P$ is a
 maximal ideal  of $\mathfrak M_1$ and
   $\mathfrak M_1/\mathfrak P$ is simple.
By Lemma \ref{lemma:lattice-quotient}(iv),
the  MV-algebra  $\Gamma(K_0(\mathfrak M_1/\mathfrak P))$
may be identified with  $\McNone/\mathfrak p$.  
With reference to \cite[\S\S 4.4-4.6]{cigdotmun} or
 \cite[\S\S 1.5, 2.5]{mun11}, there is  
 is a {\it theory}  (i.e., a set  $\Theta$ of formulas 
containing all consequences
of any finite subset of $\Theta$),  
consisting of  one-variable formulas in \luk\ logic, 
 such that $\McNone/\mathfrak p$ is
the {\it Lindenbaum algebra} of $\Theta,$
the algebra of equivalence classes of formulas
modulo $\Theta.$ 
From  the assumed recursive enumerability
of the word problem of  $\mathfrak M_1/\mathfrak P$
it follows that  $\Theta$ is recursively enumerable.
By \cite[Theorem 6.1]{mun-jfa},  the simplicity of 
$\mathfrak M_1/\mathfrak P$ entails the decidability
of the problem whether a formula
$\omega$ belongs to $\Theta$.  
By definition of Lindenbaum algebra, for  every
formula $\omega$ of  the form
$\phi\leftrightarrow \psi$ we  have: 
$\omega$ belongs to $\Theta$ iff  
${\mathsf f}_\phi/\mathfrak p$ coincides with ${\mathsf f}_\psi/\mathfrak p$  in
the quotient MV-algebra $\McNone/\mathfrak p$
= $\Gamma(K_0(\mathfrak M_1/\mathfrak P))$.  
Thus  the word problem of  
$\mathfrak M_1/\mathfrak P$ is decidable.
We have proved (b).

\bigskip
\noindent
\emph{Case 2.} 
$\mathfrak p \not= \mathfrak j_\theta$. 

By  the classification Theorem  \ref{theorem:robinson-apal},
 $\theta$ must be a rational number in $\interval$ and  
 $\mathfrak p$   either  coincides with the ideal 
$$
\mathfrak j_\theta^+=\{f\in \McNone \mid
f=0 \mbox{ on a right neighborhood of } 
\theta\}
$$
or with 
$$
\mathfrak j_\theta^-
=\{f\in \McNone \mid
f=0 \mbox{ on a left  neighborhood of } \theta\}.
$$
By  Corollary   \ref{corollary:polytime-other-bl}, the word problem
of both involutive monoids $\McNone/\mathfrak \mathfrak j_\theta^+$
and     $\McNone/\mathfrak \mathfrak j_\theta^-$ has polytime complexity. 

Having thus proved  (c),  the proof  is  complete.
\end{proof}

\begin{corollary}
\label{corollary:apal-immune}
Let  $\mathfrak B$ be a quotient  of
 $\mathfrak M_1$
with  finitely many maximal ideals, and with
a recursively enumerable word problem.
Then   the word problem of $\mathfrak B$
is decidable. 
\end{corollary}
\begin{proof}
A routine refinement of the proof of the above theorem.
\end{proof}

\begin{theorem}
\label{theorem:example} A
 quotient $\mathfrak Q$ of $\mathfrak M_1$ with a 
G\"odel incomplete word problem can be explicitly constructed.
\end{theorem}

\begin{proof}
We  refer to
\cite{man, mon} for all unexplained notions and results
 of  first order
logic needed for the proof. By Corollary, \ref{corollary:apal-immune},
$\mathfrak Q$  necessarily has    infinitely many maximal ideals. 
A fortiori,  $\mathfrak Q$  is not a
primitive quotient of $\mathfrak M_1.$ 
Let  $\mathbb G$ be the set of
G\"odel numbers of
all first-order  tautologies  (= universally true sentences)
  in the first order language of
one  binary relation
 symbol  $\in$.
  \begin{figure} 
    \begin{center}                                     
    \includegraphics[height=8.5cm]{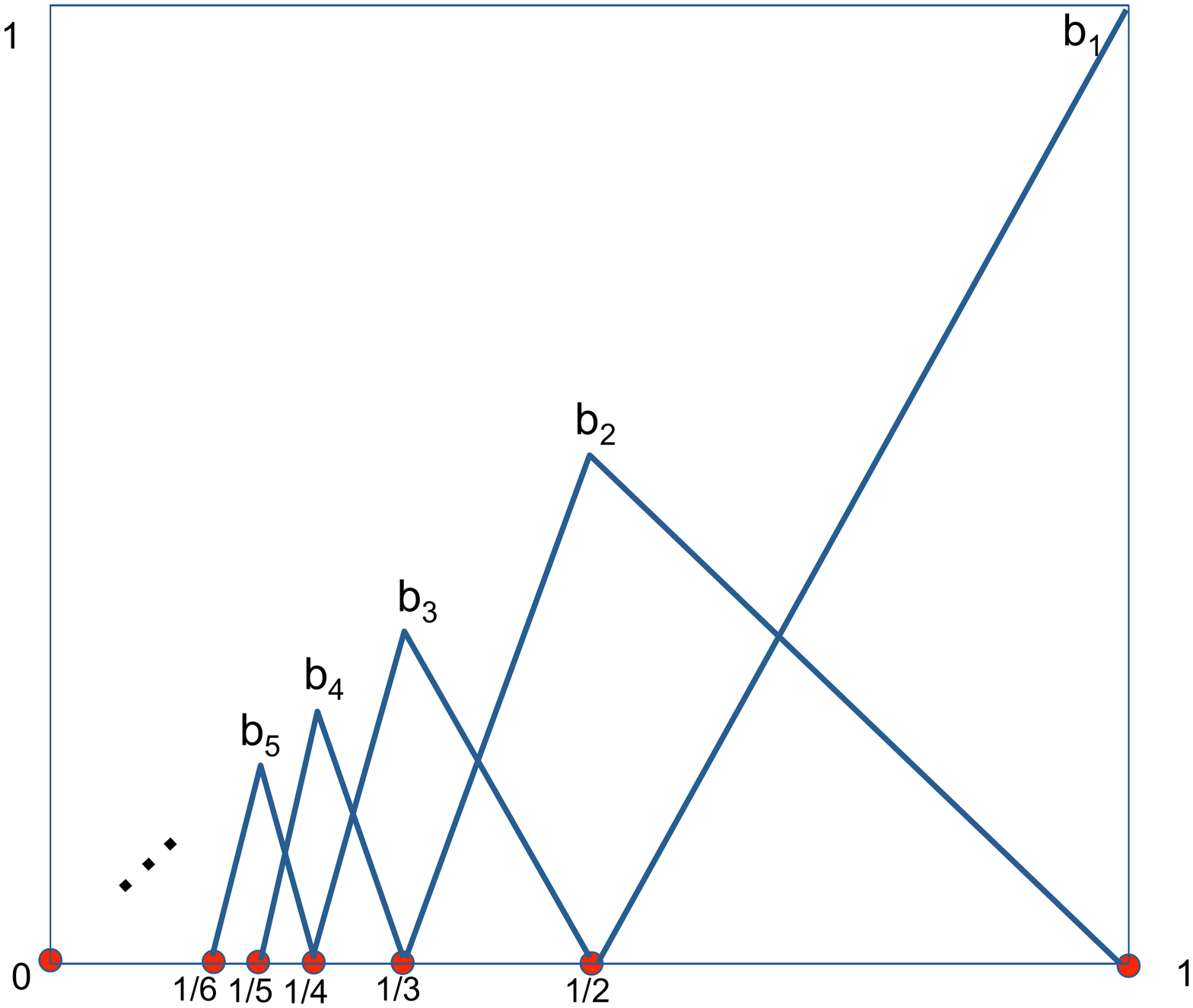}    
    \end{center}                                       
 \caption{\small The hat functions $b_k$ for $k=1,\dots,5$.
 Each linear piece of $b_k$ lies on a line
 whose equation has integer coefficients.}
    \label{figure:hats}                                                 
   \end{figure}
   We can safely assume   
 \begin{equation}
 \label{equation:01} 
0,1\notin \mathbb G. 
  \end{equation} 
On the one hand,    G\"odel's completeness theorem 
implies  that
  $\mathbb G$ is a recursively enumerable subset of
  $\{0,1,\dots\}.$
On the other  hand, 
G\"odel's incompleteness theorem
implies that $\mathbb G$  is  undecidable,
because there are undecidable finitely axiomatizable
 first-order theories
more powerful than Peano arithmetic
in the  language consisting of one binary relation symbol
$\in$.

We  define the  functions  
$b_0,b_1,\dots\colon \interval\to \interval$
as follows: 
$b_0$ is  the constant 0 function
over $\interval$. 
The graph of $b_1$ has two linear pieces, one
joining the origin with $(1/2,0)$, the other joining
$(1/2,0)$ with $(1,1).$
The graph of $b_2$ has three linear pieces, the first
joining the origin with $(1/3,0)$, the second joining
$(1/3,0)$ with $(1/2,1/2),$ and the third joining
$(1/2,1/2)$ with $(1,0).$
For each $k=3,4, \dots$,  the graph of
the hat-shaped function
$b_k $  has  four linear 
pieces: two pieces lie on the $x$-axis.
The left oblique piece joins the points
$(1/(k+1),0)$ and $(1/k,1/k)$. The right oblique piece joins the
points $(1/k,1/k)$ and $(1/(k-1),0)$.
Figure \ref{figure:hats} shows the hats $b_1,\dots,b_5$.

\medskip
\noindent
{\it Claim 1.}
There is a Turing machine which,
over input   $k=0,1,\dots,$  outputs a
 one-variable formula  $\beta(k)$ 
such that ${\mathsf f}_{\beta(k)}=b_k$.

\smallskip
If $k=0$, the formula $(X\oplus X^*)^*$
will do. For $k=1$, the formula  $\beta(1)=(X^*\oplus X^*)^*$
satisfies ${\mathsf f}_{\beta(1)}=b_1$. 
If $k=2,3,\dots$,  the  integer  $2\times 2$
matrices   with rows 
$(1, k+1)$ and $(1,k)$
is unimodular.
The same holds for the matrix 
with rows $(1,k)$ and $(1,(k-1))$. 
Since $b_k(1/k)=1/k,$ a  straightforward  verification shows that 
 the two oblique linear pieces of the graph of $b_k$ 
 are restrictions of linear polynomials with integer
coefficients. 
  McNaughton's  theorem \cite{mac} then 
merely ensures the {\it existence} of  $\beta(k)$.
To prove that  $\beta(k)$ is  {\it computable} from $k$,
in  \cite[Proposition 2.7]{aguz}  
an effective  method is given    to  compute
a formula  $\epsilon_{m,n}$  which,
for all integers  $m,n$,  codes the McNaughton
function  $e_{m,n}(x)=0\vee(1\wedge(mx+n)),\,\,\,x\in\interval.$ 
The desired formula $\beta(k)$ 
is now effectively computed as the pointwise
min of  two suitably chosen   formulas $\epsilon_{m,n}$,
recalling that,  by 
Theorem \ref{theorem:multi}(iii),
the   $\wedge$ operation  is definable
from the involutive monoidal operations
$^*$ and $\oplus$ of $\McNone$. Our first claim is settled.

\medskip
Next we fix, once and for all,    
a  recursive enumeration
$\boldsymbol{\eta}$  of $\mathbb G$.
In other words,  $\boldsymbol{\eta}\colon\{0,1,\dots\}
\to \{0,1,\dots\}$  is  recursively enumerable
and its range coincides with  $\mathbb G$.
Then
for  each $t\in \{0,1,\dots\}$ we define the formula $\gamma(t)$ by  
  \begin{equation}
 \label{equation:formulaenna}
 \gamma(t)
 =\beta({\boldsymbol{\eta}(0)})\vee
\beta({\boldsymbol{\eta}(1)})\vee\dots\vee
 \beta({{\boldsymbol{\eta}(t-1)}})
\vee \beta({{\boldsymbol{\eta}(t)}}).
  \end{equation}
Since the underlying lattice order
is definable in terms
of the $^*$ and $\oplus$ operations, and 
 $\mathbb G$ is recursively enumerable,
then so is the sequence  of formulas 
$ \gamma(0), \gamma(1),\dots$.

\smallskip
For each $t=0,1,\dots,$ let  the McNaughton function
  $g_t\in \McNone$ be defined  by 
 \begin{equation}
 \label{equation:gienne}
 g_t={\mathsf f}_{\gamma(t)}=
 {\mathsf f}_{\beta(\boldsymbol{\eta}(0))}\vee
 {\mathsf f}_{\beta(\boldsymbol{\eta}(1))}\vee\dots\vee
  {\mathsf f}_{\beta(\boldsymbol{\eta}(t-1))}
   \vee {\mathsf f}_{\beta(\boldsymbol{\eta}(t))}.
 \end{equation}
The  functions  $g_t$ form an increasing
sequence  $g_0\leq g_1\leq\dots $. 
 Let $\mathfrak g$ denote the
ideal of $\McNone$ generated by $\{g_0,g_1,\dots\}.$
For $n=1,2,\dots$ and $l\in \McNone$
let  us write
\begin{equation}
\label{equation:underbrace}
 n\,.\,l=
\underbrace{l\oplus\dots \oplus l}_{n {\rm \,\,summands}}.
\end{equation}
By definition of $\mathfrak g$, for
 any $f\in \McNone$ we then  have
$$
f \in \mathfrak g  \,\,\,\,\,\,\mbox{ iff }\,\,\,\,\,\,
 t. g_t\,\, \geq\,\, f\,\,\, \mbox{ for some } t=1,2,\dots,
 $$
because  the sequence $g_0,g_1,\dots$ is ascending.


\bigskip
\noindent
{\it Claim 2.}   The set  $\Omega$ of formulas  $\chi$  such that
${\mathsf f}_\chi$ belongs to  $\mathfrak g$
 is recursively enumerable.
 
 \smallskip
It is easy to verify that for  all $x,y\in \interval$,
\,\,$x\leq y$ iff $x^*\oplus y=1$.
A tedious but straightforward verification (\cite[p.9]{cigdotmun}) 
shows that
a  formula  $\chi$
satisfies 
 ${\mathsf f}_\chi\in \mathfrak g$
\,\,\, iff \,\,\,
$\chi^*\oplus  t.\gamma(t)$ is a tautology in \luk\ logic
for some $t$.
Since the set of tautologies in  \luk\  logic is decidable,
\cite[Theorem 4.6.10]{cigdotmun},
  \cite[18.3]{mun11},
using the recursive enumerability of the formulas
$\gamma(t)$  we can enumerate all tautologies 
 of the  form $\chi^*\oplus  t.\gamma(t)$, for some
 $t$ and $\chi$, and 
 extract from these tautologies the desired set $\Omega.$ 
Our second claim is settled.

\medskip
\noindent
{\it Claim 3.}   The set $\Omega$ 
 is undecidable.

\smallskip
By way of contradiction, suppose 
$\Omega$ is decidable.
For any  integer $k\geq 2$, 
we compute the  one-variable formula
$\beta(k)$ satisfying 
${\mathsf f}_{\beta(k)}=b_k$ as in 
the proof  of  Claim 1.  
If $k\in \mathbb G$ then
$b_k\leq g_k$,  whence $b_k\in \mathfrak g.$ 
Conversely, suppose 
 $k\notin \mathbb G$.
By \eqref{equation:gienne},
  every  $g_t$ vanishes at $1/k$,  but 
 $b_k(1/k)=1/k$. As a consequence, for no  $t=1,2,\dots$
 we may have   $t.g_t\geq b_k$,
 whence   $b_k\notin \mathfrak  g.$ 
 We have shown:
\begin{eqnarray*}
\beta(k)\in \Omega & \mbox{iff}&   b_k\in \mathfrak g\\
{}&  \mbox{iff}&  t. g_t \geq b_k \mbox{ for some } t\\
 {}&  \mbox{iff} & k\in \mathbb G,\,\,\, \mbox{(by definition of $b_k$ 
 and $g_t$).}
\end{eqnarray*}
By our absurdum hypothesis,
  membership of  $\beta(k)$  in $\Omega$
is decidable. It follows that
  for  $k= 2,3,\dots$  (whence for all $k$,
by  \eqref{equation:01})
we may effectively decide whether $k$ belongs to $\mathbb G,$
thus contradicting the undecidability of $\mathbb G$. 
Our third claim is settled.
 
\medskip
To conclude the proof of the theorem,  let
$\mathfrak G$ be the ideal of $\mathfrak M_1$  corresponding
to $\mathfrak g$ via the functors   $K_0$, 
and  $\Gamma$. 
Let the quotient AF$\ell$ algebra
$\mathfrak Q$ be defined by
$$\mathfrak Q=\mathfrak M_1/\mathfrak G.$$
By Lemma \ref{lemma:lattice-quotient}(iv),
 we can write 
$$
E(\mathfrak M_1/\mathfrak G)=  \McNone/\mathfrak g.
$$
As in   \eqref{equation:absolute},
for any two formulas
$\phi$ and $\psi$   in the language
$0,X,^*,\oplus$ of involutive monoids,    let
 $\delta$ be the formula
\begin{equation}
\label{equation:oftheform}
 ((\phi^*\oplus \psi)^*  \oplus (\psi^*\oplus \phi)^*).
\end{equation}
Then $\phi$ and $\psi$
 code the same element of
$\McNone/\mathfrak g$ 
 iff  $|{\mathsf f}_\phi - {\mathsf f}_\psi|$ belongs to
 $\mathfrak g$ iff ${\mathsf f}_\delta\in \mathfrak g.$
 From  Claim 2 it follows that
   the set  $\Delta$   of  pairs of formulas $(\phi,\psi)$ coding
 the same element of $\McNone/\mathfrak g$ 
 is recursively enumerable. 
  From  Claim 3 it  follows that
$\Delta$ 
 is  undecidable, because  formulas of the form
 \eqref{equation:oftheform} code
 all possible McNaughton functions  (letting
 $\psi=0$ in  \eqref{equation:oftheform}).

 In conclusion, the word problem of $\McNone/\mathfrak g$ 
 is G\"odel incomplete, whence so is 
  the word problem of the AF$\ell$ algebra $\mathfrak Q$.
\end{proof}

\begin{remark}
The map $ 
n\mapsto n  \mbox{th hexadecimal digit of }\pi
$
is Turing computable,  e.g., via the
Bailey-Borwein-Plouffe formula \cite{bbp},
$$
\pi = \sum_{k=0}^\infty\left[  
\frac{1}{16^k}\left(\frac{4}{8k+1} - \frac{2}{8k+4} -
\frac{ 1}{8k+5} - \frac{1}{8k+6}\right) \right].
$$
As a consequence, the rational left cut
of $\pi-3$ is decidable. As a matter of fact,
 there is a Turing machine
$\mathcal T$  which,
over input a pair of binary integers  $p,q$ with $q>p>0$
decides in a finite number of steps whether
$p/q<\pi-3.$ 
To this purpose,  $\mathcal T$ lists the hexadecimal digits
$h_1,h_2,\dots$  of
$p/q$ along with the hexadecimal digits
$k_1,k_2,\dots$   of $\pi-3$. Let $m$ be the smallest
integer such that  $h_{m}\not=k_m$. Then 
$p/q$ lies in the left cut of $\pi-3$ iff
$h_m<k_k.$

 A routine variant of the argument in the
proof of Corollary
\ref{corollary:polytime-ef-cut} now shows that {\it the
word problem of the Effros-Shen algebra 
$\mathfrak F_{\pi-3}$ is decidable.}
\end{remark} 

\begin{problem}  What is the computational complexity of
the word problem of $\mathfrak F_{\pi-3}$?
\end{problem}

\section{Concluding Remarks}
Further results on G\"odel incompleteness in the frame  of
Elliott's local semigroups  are proved in 
\cite[Example 6.4]{mun-jfa}  and  \cite[Theorem 2.3]{muntsi}.
Decidability and undecidability results
for  isomorphism problems of
AF$\ell$  algebras presented by Bratteli diagrams are
the subject matter of
 \cite{bratteli1}-\cite{bratteli2} and  \cite{mun-tams}.
 
Throughout this paper, 
the categorical equivalence
$\Gamma$ of \cite{mun-jfa} between unital $\ell$-groups
and MV-algebras
has been combined with  Grothendieck
$K_0$ functor (\cite{eff, goo-shiva}),  to investigate the
computational complexity of the word problem
of several AF$\ell$ algebras.
The natural deductive-algorithmic
 properties of the \luk\ sentential calculus
 (whose semantics is provided by  MV algebras,
 \cite{cha58,cha59})  have played
 a key role in the final section.
 
 \begin{table}[ht]
\begin{center}
\begin{tabular}{|r|l|}
\hline
{\sc from}  &  {\sc to} \\
\hline
\hline
 {\small Finitely presented MV-algebras}    &  {\small Rational polyhedra,  \cite{marspa}} \\  
        \hline
       {\small   Locally finite MV-algebras} &  
        {\small  Multisets, \cite{cigdubmun}} \\ 
                        \hline
 {\small  MV-algebras with the map
 $(x_1,x_2,\dots)\mapsto \sum_{k=1}^\infty x_k/2^k$
 }    &
 {\small  Compact  spaces, \cite{marreg}} \\  
                        \hline
\end{tabular}\\[0.2cm]
\end{center}
\caption{\small Three contravariant categorical equivalences for MV-algebras.}
\label{table:dualities}
\end{table}
Other (contravariant) categorical equivalences,
 also known as ``dualities'',
between classes of MV-algebras and various 
 classes of mathematical structures are 
 summarized in 
 Table \ref{table:dualities}, showing
 that the scope  of MV-algebras
goes  beyond \luk\ logic,
unital $\ell$ groups, and
AF$\ell$ algebras.  
 

\end{document}